\documentclass[a4paper,11pt]{amsproc}
\usepackage{amssymb}
\usepackage{amscd}
\usepackage{amsfonts}
\usepackage{amssymb}
\usepackage{cite}
\usepackage{amsmath}
\usepackage{epsfig}
\usepackage{tikz}
\usepackage{graphicx, subfigure}
\usepackage{color, graphics}
\usepackage{pdfsync}
\usepackage{hyperref}
\usepackage[all,cmtip]{xy}
\usepackage{setspace}
\usepackage[normalem]{ulem}
\theoremstyle{plain}
 \newtheorem{thm}{Theorem}[section]
 
 \newtheorem{lem}{Lemma}[section]
 
\theoremstyle{definition}
 \newtheorem{exm}{Example}[section]
 
\theoremstyle{remark}
 \newtheorem{rem}{Remark}[section]
 \numberwithin{equation}{section}

\renewcommand{\le}{\leqslant}
\renewcommand{\ge}{\geqslant}
\renewcommand{\setminus}{\smallsetminus}

\allowdisplaybreaks

\newcommand{\R}{\mathbb{ R}}

\DeclareMathOperator{\diag}{diag}

\def\mod {\mathrm{mod}}

\def\diag{\mathrm{diag}}
\def\max{\mathrm{max}}

\def \i{{\rm i}}

\def\px1{p_{x_1}}
\def\px2{p_{x_2}}
\def\pu1{p_{u_1}}

\title[Bungee and $C^1$ and $C^0$ Hamiltonian systems]{From bungee to $C^1$ and $C^0$ Hamiltonian systems and their integrability and nonintegrability}

\subjclass[2020]{37J35, 37E40, 70H06, 70H07, 37C83, 34A38}

\keywords{$C^0$ and $C^1$ Liuoville-Arnol'd theorem, gluing of integrable systems, Poincar\'e section, Moser's twist}

\author[V. Dragovi\'c, B. Gaji\'c, and B. Jovanovi\'c]{\bfseries Vladimir Dragovi\'c, Borislav Gaji\'c, and Bo\v zidar Jovanovi\'c}

\address{
Department of Mathematical Sciences, The University of Texas at Dallas,  USA
\\
Mathematical Institute, Serbian Academy of Sciences and Arts, Belgrade, Serbia}
\email{Vladimir.Dragovic@utdallas.edu}

\address{
Mathematical Institute, Serbian Academy of Sciences and Arts, Belgrade, Serbia}
\email{gajab@mi.sanu.ac.rs}

\address{
Mathematical Institute, Serbian Academy of Sciences and Arts, Belgrade, Serbia}
\email{bozaj@mi.sanu.ac.rs}

\begin{document}

\begin{abstract} We consider natural Hamiltonian systems with potentials that are $C^0$ or $C^1$ on a hypersurface and $C^{\infty}$-smooth in the complement  and introduce and study corresponding notions of  their integrabilty and non-integrability. As a motivating example, we derive and  analyze  models of bungee jumping. We provide  prototype examples of the Liuoville-Arnol'd theorem for $C^0$ and $C^1$ Hamiltonians.
\end{abstract}
\maketitle

\section{Introduction}

This work is motivated by bungee jumping. One very interesting modeling of bungee jumping, as an infinite-dimensional system, was given in \cite{LLM}. Here, we propose a different approach. We construct classical mechanical models with almost everywhere $C^\infty$-Hamiltonian functions that are of the class $C^1$ and $C^0$ on a given hypersurface $\Delta$.  By restricting the continuous in $t$ Hamiltonian dynamics, to a discrete dynamics on $\Delta$, we introduce and describe a return map,  an analog of a billiard map. The Hamiltonian systems that we construct in this paper are close to hybrid systems from \cite{CB} and
to magnetic billiards considered in \cite{Ga2021, Ga2025}  (see also \cite{GaRa2024}). The generated dynamics lies somewhere between usual smooth Hamiltonian systems and the refraction billiards from \cite{dBT, dBT2, BCdB}  and in spirit is also close to the so-called port-Hamiltonian systems, see e.g. \cite{CvdS}.

Let us note that a more realistic model of bungee would include air resistance producing the drag  force. However, such a system would be dissipative and thus out of scope of the current paper.

 Having that in mind, we first model bungee jumping in $\R^2\{x_1,x_2\}$, as a modification of the elastic pendulum. We define it as
a Hamiltonian system on $T^*\R^2$
with the Hamiltonian
\begin{equation}\label{eq:ham*}
H(x,p)=\frac12(p_1^2+p_2^2) +V^*(x), \qquad V^*(x)=g x_2+\frac12\sigma (\vert x\vert-\ell)^2\theta(\vert x\vert-\ell).
\end{equation}
Here $\theta$ is the Heaviside step function
\begin{equation*}
\theta(x) =
    \begin{cases}
        1, & \text{if } x>0,\\
        1/2, & \text{if } x=0,\\
        0, & \text{if } x<0,
    \end{cases}
\end{equation*}
and $\sigma, g>0$ are real parameters,  the coefficient of elasticity of the spring and the gravitational constant, respectively.
Thus, within the disk
\begin{equation}\label{eq:disk}
D=\{x=(x_1,x_2)\in\R^2 \,\vert\, \langle x,x\rangle=x_1^2+x_2^2 \le {\ell^2}\},
\end{equation}
centered at the origin $O(0,0)$ and of radius $\ell$, the system describes the motion of a particle of the unit mass in a homogeneous gravitational field. Outside the disk \eqref{eq:disk}, the motion is under the influence of an additional elastic spring.
By replacing the Heaviside function with a constant, one gets the elastic pendulum in a homogeneous gravitation field \cite{Ku}.
Using the smooth approximation of the Heaviside function
$$
\theta(x)=\lim_{k\rightarrow \infty} \theta_k(x),
$$
where
$$\theta_k(x)=\frac{1}{2} +\frac{1}{2}\tan kx,$$
we can consider the family of smooth Hamiltonians
$$
H_k(x,p)=\frac12 \langle p,p\rangle +V_k^*(x), \qquad V_k^*(x)=g x_2+\frac12\sigma (\vert x\vert-\ell)^2\theta_k(\vert x\vert-\ell),
$$
that in a limit produce our Hamiltonian $H=\lim_{k\rightarrow \infty} H_k$.
This construction generalizes to arbitrary dimensions in a straightforward fashion.

The Hamiltonian function \eqref{eq:ham*} is $C^1$-smooth on $T^*\R^2$ and $C^\infty$-smooth outside the hypersurface
\begin{align}\label{Delta}
%&\Delta_0=T^*_{\partial D}\partial D=\{(x,p)\,\vert\, \langle x,x\rangle=1,\, \langle p,x\rangle=0\}, \\
\Delta= T^*_{\partial D} \R^2=\{(x,p)\, \vert\, \langle x,x\rangle=x_1^2+x_2^2=\ell^2\}.
\end{align}

There is a natural identification of cotangent and tangent bundles with respect to the Euclidean metric. It gives a decomposition of $\Delta$ on a part tangent to the circle $\partial D$, a part with outgoing velocities $\Delta_+$, and a part with ingoing velocities $\Delta_-$:
\begin{align}
\label{Delta0}&\Delta_0=T^*\partial D=\{(x,p)\,\vert\, \langle x,x\rangle=\ell^2,\, \langle p,x\rangle=0\},\\
\label{Delta+}&\Delta_+=\{(x,p)\,\vert\, \langle x,x\rangle=\ell^2,\, \langle p,x\rangle>0\},\\
\label{Delta-}&\Delta_-=\{(x,p)\,\vert\, \langle x,x\rangle=\ell^2,\, \langle p,x\rangle<0\},\qquad  \Delta=\Delta_0\cup \Delta_+\cup\Delta_-.
\end{align}

The Hamiltonian vector field is  Lipschitz continuous implying  the uniqueness of the  solutions $(x(t),p(t))$ of the Hamiltoinian system, determined by the initial data. The solutions $(x(t),p(t))$  of the Hamiltonian system are $C^\infty$-smooth at $t_0$ for all $t_0$ such that  $(x(t_0),p(t_0))\notin\Delta$. Moreover, the trajectories $(x(t),p(t))$ that are tangent to $\Delta_0$
at the moment $t_0$ and  for which there exists $\epsilon>0$ such that $(x(t),p(t))$ either, for all  $t\in (t_0-\varepsilon,t_0+\varepsilon)$ belongs to the
region $\vert x(t)\vert \ge\ell$, or for all  $t\in (t_0-\varepsilon,t_0+\varepsilon)$ to the region $\vert x(t)\vert \le  \ell$, are also $C^\infty$-smooth at $t_0$.

The elastic pendulum was very well studied
in dimensions two and  three   (see e.g. \cite{Ku, L, MP, AMBC} and references therein). In particular, the elastic pendulum is not integrable. This indicates that the Hamiltonian system defined with the Hamiltonian function \eqref{eq:ham*} is not integrable as well.

As the next step in Section \ref{sec2}, we consider a simplified model without the gravitational field.
We will show that this simplified model is an example $C^1$-integrable system in Theorem \ref{th:integrability}. We will also describe the corresponding smooth return mappings, a discrete system obtained as the restriction of the continuous
Hamiltonian dynamics in $t$ to the Poincar\'e mappings of the hypersurfaces $\Delta_+$ and $\Delta_-$, see Theorem \ref{th:billiard}.

In Section \ref{sec3}  we consider $C^0$-modification  of the system \eqref{eq:ham*}, which is going to be defined by the Hamiltonian
\begin{equation}\label{eq:ham2}
H(x,p)=\frac12 \langle p,p\rangle +V(x), \qquad V(x)=gx_2+\frac12\rho (\vert x\vert^2-{\ell^2})\theta(\vert x\vert-\ell).
\end{equation}
Now, the Hamiltonian vector field is not Lipschitz continuous at $\Delta$. We can see the modified system as a $C^0$-gluing along $\Delta$
of two super-integrable systems: the motion in a homogeneous gravitational field in $\R^2$ with noncompact invariant manifolds and the motion under the influence of the Hook potential centered at $C(0,-{g}/{\rho})$ with compact invariant manifolds.
Both flows are transversal to $\Delta_+\cup\Delta_-$, where the dynamics can be naturally defined
(Section \ref{sec3}).
However,  there exists a singular set $\Pi\subset \Delta_0$, where the dynamics is not defined.
The trajectories that intersects $\Delta_+\cup\Delta_-$ define return maps that we explicitly derive in a closed algebraic form for the value of the energy $H=h<\rho\ell^2$
(Theorems \ref{th:prva} and \ref{th:druga}). The explicit form of the Poincar\'e map allows us to obtain  a Poincar\'e section (see Figure \ref{fig:haos}) indicating that the $C^0$-approximate system is not integrable. Exploiting further this explicit formula for the Poincar\'e map, we also prove that the origin is a non-degenerate non-resonant up to order four elliptic fixed point with the nonzero Moser twist coefficient. This implies that the origin of the Poincar\'e section is Lyapunov stable and the majority of the trajectories in its neighborhood belong to invariant circles (see \cite{Kr} for the background theory).

In Section \ref{sec4}, we  consider a special case of \eqref{eq:ham2} with $g=0$. We extend the dynamics for all points of $\Delta_0$ in this case.
As a result, we obtain a $C^0$-gluing of two super-integrable systems, to which we apply  a  $C^0$-version of the Liouville-Arnol'd theorem
and derive the associated integrable return maps  (Theorem \ref{integrabilniBilijar}).

 Note that the
$C^1$-Liouville-Arnol’d theorem was formulated recently in \cite{AX}, with a $C^2$-smooth  Hamiltonian and $C^1$-smooth first integrals. Also, there is a notion of $C^0$-integrability of Hamiltonian systems \cite{AABZ} and symplectic mappings \cite{BM, Si}. However, the classical mechanical examples of
 the Liouville-Arnol’d theorem with $C^1$ and $C^0$ Hamiltonian functions have not been studied yet. That is why
the above problems motivate us to introduce a novel concept of $C^0$-integrability with a $C^0$-Hamiltonian, that we do in Section \ref{sec5}.

\section{Construction of a $C^1$-integrable system}\label{sec2}

\subsection{A bungee without a gravitational field}

As the next step, we consider a planar case of the above Hamiltonian without a homogeneous gravitation field, thus assuming  $g=0$ in \eqref{eq:ham*}. The Hamiltonian function of this simplified system is
\begin{equation}\label{eq:ham}
H(x,p)=\frac12 \langle p,p\rangle+\frac12\sigma (\vert x\vert-\ell)^2\theta(\vert x\vert-\ell).
\end{equation}

Within the disk $D$ (see  \eqref{eq:disk}), the motion of the new Hamiltonian system is uniform,
\begin{equation}\label{ham1}
\dot x=p, \qquad \dot p=0,
\end{equation}
while outside the disk $D$, the motion is in the central force filed
\begin{equation}\label{ham2}
\dot x=p, \quad \dot p=\mathbf F, \quad \mathbf F(x)=-\sigma \frac{(\vert x\vert -\ell)}{\vert x\vert} x=-\sigma x+\sigma \ell \frac{x}{\vert x\vert} \quad \big(\mathbf F\vert_{\partial D}\equiv 0\big).
\end{equation}

Both regimes as independent systems are well-known, see e.g \cite{Ar}, and known to be integrable. The purpose of this study is to focus on the hybrid system, obtained by gluing together these two system. The hybrid system is also solvable by quadratures, due to the Noether first integral, that comes from the $SO(2)$-symmetry:
\[
S(x,p)=x_1 p_2-x_2 p_1;
\]
$S$ is known as {\it the area first integral} (see \cite{Ar}). Consider a nonzero level of the area integral $S(x,p)=s\ne 0$. In the polar canonical coordinates
\begin{align}
\label{polarne}&x_1=r\cos\varphi, \, x_2=r\sin\varphi,  \, p_1=p_r \cos\varphi- \frac{p_\varphi}{r}\sin\varphi, \, p_2=p_r \sin\varphi+ \frac{p_\varphi}{r}\cos\varphi,\\
\label{kanonska}&\omega=dp_1\wedge dx_1+dp_2\wedge dx_2=dp_r\wedge dr+dp_\varphi\wedge d\varphi,
\end{align}
the Hamiltonian \eqref{eq:ham}  takes the form
\[
H(r,\varphi,p_r,p_\varphi)=\frac12 \big(p_r^2+\frac{1}{r^2}p_\varphi^2\big)+\frac12\sigma (r-\ell)^2\theta(r-\ell).
\]

In these coordinates, the area first integral writes as $S=p_\varphi=r^2\dot\varphi$ and it corresponds to the cyclic coordinate $\varphi$. At the invariant non-zero level set $S=s\ne 0$,
after the elimination of the cyclic coordinate, which can be seen as an $SO(2)$-symplectic reduction, we get a one-dimensional reduced system
on $T^*\R_+$ with the Hamiltonian
\[
H_s(r,p_r)=\frac12 p_r^2+\frac{s^2}{2r^2}+\frac12\sigma (r-\ell)^2\theta(r-\ell).
\]
In the region $r\ge\ell$, the reduced Hamiltonian $H_s$ is
\[
H_s(r,p_r)=\frac12 p_r^2+V_s(r), \quad V_s(r)=\frac{1}2 \big(\frac{s^2}{r^2}+\sigma r^2-2\sigma r\ell+\sigma {\ell^2}\big).
\]
For $0<r\le\ell$, the same Hamiltonian is
\[
H_s(r,p_r)=\frac12 p_r^2+V_s(r), \quad V_s(r)=\frac{s^2}{2r^2}.
\]

Therefore, the \emph{effective potential} $V_s(r)$ is  $C^{\infty}$-smooth on $(0,\ell)\cup (\ell,\infty)$ and it is of the class $C^1$ on $(0,\infty)$  (see Figure \ref{fig:efektivni}). For the second derivatives at $\ell$, one gets
\[
\lim_{r\to \ell-}V''_s(r)=3\frac{s^2}{\ell^4}, \qquad \lim_{r\to \ell+}V''_s(r)=3\frac{s^2}{\ell^4}+\sigma.
\]

The reduced system was  described in detail in \cite{Ar}.
 Since $V''_s(r)>0$ for $r\ne \ell$, $\lim_{r\to\infty} V_s(r)=\infty$ and
$\lim_{r\to 0} V_s(r)=\infty$, the effective potential $V_s(r)$ is strictly convex and it has a unique minimum $r^*_s$. The  trajectories in $T^*\R_+$ are the cycles
 of the form $\gamma_h=\{H_s=h\}$
for every $h>V_s(r^*_s)$ and the equilibrium
$\{(r^*_s,0)\}$  (see  Figure \ref{fig:efektivni}).
Note that $r^*_s$ tends to $\ell$ and $V_s(r^*_s)$ tends to zero, as $s$ tends to zero.

Define $\pi_r\colon T^*\R_+\to \R_+$, $\pi_r(r,p_r)=r$.
Then,  $\pi_r(\gamma_h)=[r_{\min,s},r_{\max,s}]$ is the projection of the cycle $\gamma_h$.
The values $r_{\min,s}$ and $r_{\max,s}$ correspond to the points on $\gamma_h$ where the kinetic energy is equal to zero.

We have the following two cases: (i) If $V_s(\ell)={s^2}/({2\ell^2})>h$, then  $r_{\min,s}>\ell$. In this case  $r_{\min,s}$ and $r_{\max,s}$ are the solutions of the  equation
\begin{equation}\label{rMinMax}
\frac12\big(\frac{s^2}{r^2}+\sigma r^2-2\sigma \ell r+\sigma {\ell^2}\big)=h,
\end{equation}
for $r>0$.  (ii) If $V_s(\ell)={s^2}/({2\ell^2})\le h$, then  $r_{\min,s}\le \ell$. It is the solution of the equation ${s^2}/{2r^2}= h$ for $r>0$:
\[
r_{\min,s}=\frac{\vert s\vert}{\sqrt{2h}},
\]
while $r_{\max,s}$ is the unique solution
of the equation \eqref{rMinMax} greater than $\ell$.  In particular, for
$V_s(\ell)={s^2}/({2\ell^2})=h$, we have $r_{\min,s}=\ell$.

The period of all the trajectories that belong to $\gamma_h$ is
\begin{equation}\label{period}
T_{h,s}=2\int_{r_{\min,s}}^{r_{\max_s}} \frac{dr}{\sqrt{2(h-V_s(r))}}.
\end{equation}

Each of those trajectories in the original configuration space $\R^2\{x_1,x_2\}$ belongs to the annulus between  the circles of radii $r_{\min,s}$ and $r_{\max,s}$, centered at the origin $O(0,0)$  (see Figure \ref{fig:prsten}). Note that in the case (i), $r_{\min,s}>\ell$ and the trajectories do not intersect $\Delta$.

Since $s=r^2\dot\varphi$, during the period $T_{h,s}$, the angle coordinate increases for
\begin{equation}\label{prirastaj}
\Delta\varphi_{h,s}=\varphi(T_{h,s})-\varphi(0)=\int_0^{T_{h,s}} \frac{s}{r^2(t)}dt=2\int_{r_{\min,s}}^{r_{\max,s}} \frac{s\, dr}{r^2 \sqrt{2(h-V_s(r))}}.
\end{equation}
 In Figure \ref{fig:prsten}, the angle $\Delta\varphi_{h,s}$ is equal to the angle $\angle \mathbf x_0 O \mathbf x_2$.

For $S=0$, the trajectories are either equilibria ($x\in D$, $p=0$) or periodic motion along the lines through the origin. For example, consider the invariant plane $\{p_2=0, x_2=0\}$  and a periodic motion $(x^0_1(t),p^0_1(t))$ with the energy $H=h>0$. Then $x_1^0(t)\in [x_{1,\min},x_{1,\max}]$, where
\[
x_{1,\max}=\sqrt{2\frac{h}{\sigma}}+\ell, \qquad x_{1,\min}=-x_{1,\max},
\]
 and the period is:
\[
T_h=2\int_{x_{1,\min}}^{x_{1,\max}} \frac{dx_1}{\sqrt{2h-\sigma (\vert x_1\vert-\ell)^2\theta(\vert x_1\vert-\ell)}}.
\]

All other trajectories, up to a time translation, with the zero value of the area first integral $S=0$ and the same energy $H=h$ are given by rotations  $\mathbf R_\vartheta$ about the origin
of the given periodic motion, for an arbitrary angle $\vartheta$:
\begin{align}\label{eq:rot}
& x_1(t)= \cos\vartheta\, x_1^0(t), \qquad   x_2(t)=\sin\vartheta\, x_1^0(t), \\
& p_1(t)= \cos\vartheta\, p_1^0(t), \qquad\, p_2(t)=\sin\vartheta\, p_1^0(t).
\end{align}

The Hamiltonian vector field is of the class $C^0$ at $\Delta$.
From the structure of the equations \eqref{ham1} and \eqref{ham2}, we see that for a given trajectory $(x(t),p(t))$ of the Hamiltonian system,
the function $x(t)$ is of the class $C^2(\R)$ and $p(t)=\dot x(t)$ is of the class $C^1(\R)$  at $\Delta$. The functions $x(t)$ that are tangent to $\partial D$, corresponding to the case $r_{\min,s}=\ell$, or entirely
belong to the region $\vert x\vert >\ell$, when $r_{\min,s}>\ell$, are  $C^{\infty}$-smooth.

\subsection{The return maps}\label{bil}
Next, we describe the trajectories such that $x(t)$ has two regimes, a free motion within the disk $D$  (see \eqref{eq:disk}), and a motion under the influence of the central force field $\mathbf F$ outside the disk $D$. We will associate to them return maps on the hypersurfaces  $\Delta_+$ and $\Delta_-$, that are defined in \eqref{Delta+} and \eqref{Delta-}.

Consider a trajectory $(x(t),p(t))$, such that $S=s\ne 0$, $x(0)= \mathbf x_0\in\partial D$ and $p(0)= \mathbf p_0$ is directed outside the disk $D$, thus $r(0)=\ell$, and $p_r(0)>0$. The energy $h$ can be calculated at the initial moment $t=0$
\begin{equation}\label{pocetnaEnergija}
h=H(\mathbf x_0,  \mathbf p_0)=H_s(r(0),p_r(0))=\frac12 p_r(0)^2+\frac{s^2}{2{\ell^2}}.
\end{equation}
After the time $\tau_0$,
\[
\tau_0=\int_0^{\tau_0} dt= \int_\ell^{r_{\max}} \frac{dr}{\sqrt{2(h-V_s(r))}},
\]
the trajectory reaches the maximal radius $r_{\max}$, that is the unique solution of the equation \eqref{rMinMax}
for $r>\ell$. For $\tau=2\tau_0$, $\mathbf x_1=x(\tau)$ again belongs to the circle $\partial D$, with the velocity $ \mathbf p_1=p(\tau)$ directed within the disk,
such that $\vert \mathbf p_1\vert =\vert \mathbf p_0\vert$ and $p_r(\tau)=\dot r(\tau)=-\dot r(0)=-p_r(0)$. As in \eqref{prirastaj}, we get
\[
\varphi(\tau_0)-\varphi(0)=\int_0^{\tau_0} \frac{s}{r^2(t)}dt.
\]

Thus, for $\tau=2\tau_0$, $\mathbf x_1$ is obtained from $ \mathbf x_0$ by  the rotation around the origin for the angle:
\begin{equation}\label{theta}
\vartheta=2(\varphi(\tau_0)-\varphi(0))=2\int_\ell^{r_{\max,s}} \frac{s\, dr}{r^2 \sqrt{2(h-V_s(r))}}.
\end{equation}

As a result, we obtain that
\begin{equation}\label{refleksija}
( \mathbf x_1, \mathbf p_1)=\big(\mathbf R_\vartheta( \mathbf x_0),\mathbf S_{(\mathbf R_\vartheta( \mathbf x_0))}
\circ \mathbf R_\vartheta( \mathbf p_0)),
\end{equation}
where $\mathbf R_\vartheta$ (see \eqref{eq:rot}) is the rotation about the origin for the angle $\vartheta$ given by \eqref{theta}
 (see Figure \ref{fig:prsten}, the angle \eqref{theta} is equal to the angle $\angle \mathbf x_0 O \mathbf x_1$).
 Here $\mathbf S_x$ is the reflection at $x\in\partial D$ with respect to the tangent to the circle $\partial D$, given by the formula:
\[
\mathbf S_x(p)=p-2\frac{\langle x,p\rangle }{\langle x,x\rangle}x.
\]

Let $\Sigma_-\subset T^*\R^2$ be the open region consisting of the above trajectories, including the trajectories with $S=0$ and $H=h>0$.
%that is the union of all trajectories
%with the initial conditions $x(0)$ belonging to the interior of the disk $D$ and $p(0)\ne 0$.
It is defined by:
\begin{equation}\label{sigma-}
 \Sigma_-=\{(x,p)\in T^*\R^2\,\vert\,  {\vert s\vert}<\ell{\sqrt{2h}}\}.
\end{equation}

Note that $\Delta \cap \Sigma_-=\Delta_+ \cup \Delta_-$ (see \eqref{Delta+} and \eqref{Delta-}) and the Hamiltonian flow induces the mappings
\begin{align}
\label{Psi} & \Psi\colon  \Delta_+ \to \Delta_-, \\
\label{Phi} & \Phi\colon  \Delta_- \to \Delta_+.
\end{align}
For $S=s\ne 0$, $\Psi$ is defined using equation \eqref{refleksija}
\[
(\mathbf x_0, \mathbf p_0)\longmapsto (\mathbf x_1,\mathbf p_1)=\Psi(\mathbf x_0, \mathbf p_0),
\]
while $\Phi$ is given by
\begin{align}
\label{PhiC1}  & (\mathbf x_1, \mathbf p_1)\longmapsto (\mathbf x_2,\mathbf p_2)=\Phi(\mathbf x_1, \mathbf p_1), \\
\nonumber &    (\mathbf x_2,\mathbf p_2)=(\mathbf x_1+ \tau_1 \mathbf p_1,p\mathbf p_1),  \qquad
 \tau_1=-{2\langle \mathbf x_1,\mathbf p_1\rangle }/{\langle \mathbf p_1,\mathbf p_1\rangle }.
\end{align}
The time $\tau_1$ is determined from the condition that $ \mathbf x_2=\mathbf x_1+\tau_1 \mathbf p_1$ belongs to the circle $\partial D$  (see Figure \ref{fig:prsten}).

In addition, for $S=0$, we  have:
\begin{align*}
&\Psi(x,p)=(x,-p), \quad  p=\lambda x, \quad \lambda>0, \\
&\Phi(x,p)=(-x,p), \quad  p=\lambda x, \quad \lambda<0.
\end{align*}

As a result, we introduce  $C^\infty$-smooth return maps,
\[
\Theta_+=\Phi\circ\Psi\colon \Delta_+ \to \Delta_+, \qquad \Theta_-=\Psi\circ\Phi\colon \Delta_- \to \Delta_-.
\]
The map $\Theta_+$ is an analogous to \emph{the first return map} in \cite{dBT}.

 Note that we could have defined the return maps directly, as the rotations by the angle $\Delta\varphi_{h,s}$ for $S=s\ne 0$ (see \eqref{prirastaj} and Figure \ref{fig:prsten}).
 The restriction of the Hamiltonian $H$ to $\Delta$ is the kinetic energy. Along with $H$ and $S$,  as rotations, the return maps $\Theta_\pm$ also preserve the scalar product $J(x,p)=\langle x,p\rangle$, which is not a first integral of the continuous system. We have $S^2+J^2=2{\ell^2}\langle p,p\rangle$.

Consider the isoenergetic level sets of the Hamiltonian $H$
\begin{equation}\label{fiksiranaE}
M_{h}^+=\Big\{H=\frac12\langle p,p\rangle=h\Big\}\cap \Delta_+, \qquad M_h^-=\Big\{H=\frac12\langle p,p\rangle=h\Big\}\cap \Delta_-,
\end{equation}
diffeomorphic to the cylinder $S^1\times (0,1)$.  Let $\Theta_\pm^h$ be the restriction of $\Theta_\pm$ to $M_h^\pm$.
The canonical symplectic form $\omega=dp_1\wedge dx_1+dp_2\wedge dx_2$ induces
the volume forms $\Omega_+$ and $\Omega_-$ on $M_h^+$ and $M_h^-$, respectively, that are preserved by the mappings $\Theta_\pm^h$.

\begin{thm}\label{th:billiard}
The return maps $\Theta_\pm^h\colon M_h^\pm\to M_h^\pm$ are integrable.  For the value of the integral $S=s\ne 0$, the return maps are the rotations by the angle  $\Delta\varphi_{h,s}$ given by \eqref{prirastaj}. For $S=0$, they are equal to the involution $(x,p)\mapsto (-x,-p)$.
\end{thm}
\begin{figure}[h]
{\centering{\includegraphics[width=7.5cm]{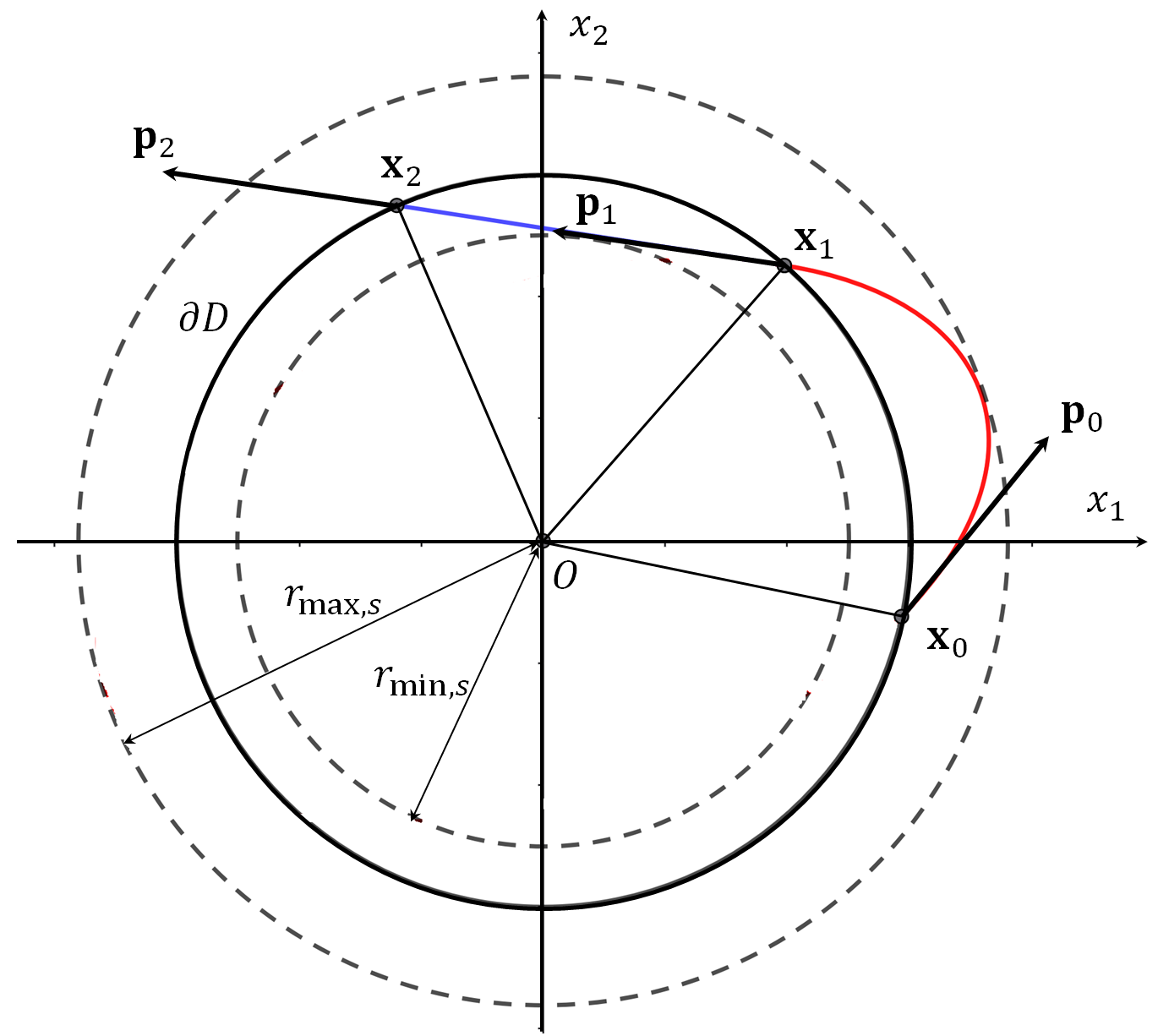}}
\caption{The return map $\Theta_+\colon  (\mathbf x_0, \mathbf p_0)  \longmapsto (\mathbf x_1,\mathbf p_1)\longmapsto(\mathbf x_2, \mathbf p_2)$. \label{fig:prsten}
}}\end{figure}
\begin{rem}\label{S=s}
At $M_h^\pm$, we have that $S=x_1p_2-x_2p_1=\cos(\beta)\ell \vert p\vert=\cos(\beta)\ell \sqrt{2h}$, where $\beta$ is the angle between the outgoing for $M_h^+$ and ingoing for $M_h^-$ velocity $p$ and the circle $\partial D$ with respect to the standard orientation $dx_1\wedge dx_2$. The sets $M_h^\pm$ are foliated on the circles
\[
\delta_{h,s}^+=M_h^+ \cap \{S=s\}, \qquad \delta_{h,s}^-=M_h^- \cap \{S=s\}
\]
and the dynamics is simply the rotation  by the angle $\Delta\varphi_{h,s}$ on $\delta^\pm_{h,s}$, where $\Delta\varphi_{h,0}:=\pi$. The boundaries of $M_h^\pm$  belong to $\Delta_0$ and they correspond to the angles $\beta=0$ and $\beta=\pi$.
\end{rem}

\subsection{$C^1$-integrability in the entire phase space}
Although the Hamiltonian $H$ from \eqref{eq:ham} is of the class $C^1$, from the above analysis it follows that the Liouville-Arnol'd theorem still can be applied.
The procedure is quite standard, as described in e.g. \cite{Ar}, with a caveat that the  action-angle coordinates obtained here are given by  a $C^1$-diffeomorphism,  which is $C^\infty$-smooth outside \eqref{Delta}.

Let $\Sigma_+\subset T^*\R^2$ be an open region consisting of trajectories $(x(t),p(t))$, such that
$x(t)$  does not intersect $\partial D$: $\vert x(t)\vert>\ell$, $t\in\R$.
It is described by the inequality:
\begin{equation*}\label{sigma+}
\Sigma_+=\{(x,p)\in T^*\R^2\,\vert\,  {\vert s\vert}>\ell{\sqrt{2h}}\}.
\end{equation*}
It is clear that $\Sigma_+$ is an invariant manifold of our system and that within $\Sigma_+$ we have the usual $C^\infty$-smooth Liouville-Arnol'd integrability.

Let $\Sigma_0=\{H=0\}=\{(x,0) \, \vert\, x\in D\}$ be the set of all equilibrium points and let
\begin{equation*}\label{sigma}
\Sigma=\{(x,p)\in T^*\R^2\,\vert\,  {\vert s\vert}=\ell{\sqrt{2h}}, h>0\},
\end{equation*}
be the collection of all the trajectories  that are tangent to the boundary circle $\partial D$.
Note that
$\Sigma_0$ and $\Sigma$ are $2$ and $3$-dimensional invariant varieties of  the system.
Thus, we obtain  a decomposition of the phase space  (see  \eqref{sigma-}), invariant with respect to the Hamiltonian flow of $X_H$,
 with  the $C^1$-Hamiltonian $H$ given by \eqref{eq:ham}:
\begin{equation}\label{razbijanje}
T^*\R^{2}=\Sigma_0\cup\Sigma_- \cup \Sigma \cup \Sigma_+.
\end{equation}

The regular invariant tori of the Hamiltonian system are given with $s\ne 0$ and $h>0$ such that $h>V_s(r^*_s)$, where $r^*_s$ is the minimum of the effective potential $V_s$, by the equations $H=h>0$ and $S=s\ne 0$.
Denote the set of those regular values $(h,s)$ by ${\mathcal Reg}$.

Firstly, following \cite{Ar}, we describe the action-angle coordinates for the one-dimensional reduced system $(T^*\R_+, H_s)$, $s\ne 0$.
The action variable is defined by:
\begin{equation}\label{I}
I(H_s)\vert_{H_s=h}=\frac1{2\pi} \int_{\gamma_h} p_r dr=\frac{1}{\pi}\int_{r_{\min,s}}^{r_{\max,s}} \sqrt{2(h-V_s(r))}dr,
\end{equation}
where, as above, $\gamma_h$ is the cycle given by the equation $H_s(p_r,r)=h>V_s(r^*_s)$, where, again, $r^*_s$ is the minimum of the function $V_s(r)$.
The value $I(H_s)\vert_{H_s=h}$ is the area of the region within $\gamma_h$ divided by $2\pi$.

Consider the ray $\Gamma_s=\{(r^*_s,p_r)\, \, p_r>0\}\subset T^*\R_+$.
The angle variable $\psi\, (\mod\,2\pi)$  on $T^*\R_+\setminus \{(r^*_s,0)\}$, which satisfies
\[
dp_r \wedge dr=dI \wedge d\psi,
\]
with the condition $\psi\vert_{\Gamma_s}=0\,(\mod\,2\pi)$, is geometrically defined by
\begin{equation}\label{ugao}
\psi(r,p_r)=\lim_{\Delta I\to 0}\frac{\int_{\gamma'_{\hat h}} p_r dr-\int_{\gamma'_{h}} p_r dr}{\Delta I}=
\lim_{\Delta I\to 0}\frac{\int_{\Pi_{\hat h,r,h}}dp_r\wedge dr}{\Delta I}, \quad (r,p_r)\in\gamma_h.
\end{equation}
Here $h=H(r,p_r)$;  the value of the energy $\hat h$ is chosen such that
the difference of the action variables $I(H_s)\vert_{H_s=\hat h}-I(H_s)\vert_{H_s=h}$, which is the area of the region between the curves
$\gamma_h=\{H=h\}$ and $\gamma_{\hat h}=\{H=\hat h\}$  divided by $2\pi$, is equal to $\Delta I$;
$\gamma_h'$ is a part of $\gamma_h$,  starting at $\Gamma_s$ and finishing at $(r,p_r)$; $\gamma_{\hat h}'$ is a part of $\gamma_{\hat h}$, starting at $\Gamma_s$ and finishing at $(r,\hat p_r)$; and
$\Pi_{\hat h,r,h}$ is the area of the  region bounded by: $\gamma_{\hat h}'$, the line that projects to $r$, $\gamma'_h$, and $\Gamma_s$ (see Figure \ref{fig:efektivni}).
\begin{figure}[h]
{\centering{\includegraphics[width=10.5cm]{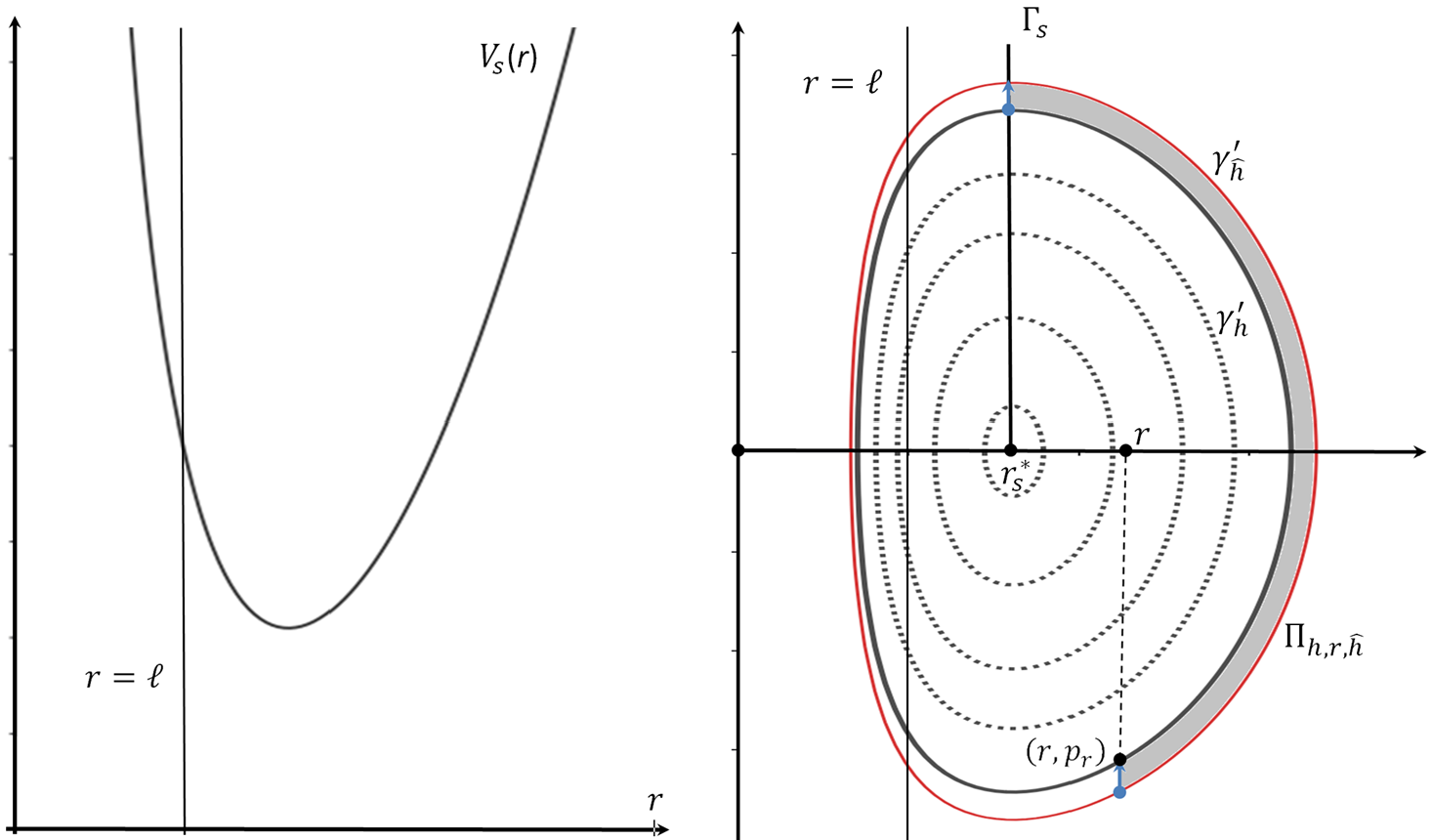}}
\caption{The effective potential and the phase of the reduced one-dimensional system. \label{fig:efektivni}
}}\end{figure}
The function $I$ is $C^{\infty}$-smooth  and monotonic in the variable $h=H_s$. We can express $H_s$ as a $C^\infty$-smooth monotonic function of $I$ as well.
The map
\[
I=I(H_s)=I(r,p_r), \quad \psi=\psi(r,p_r)\, \mod \, 2\pi,
\]
is a  $C^1$--diffeomophism between $(T^*\R_+\setminus \{(r^*_s,0)\})\{r,p_r\}$ and $S^1\times \R_+\{\psi\,(\mod\,2\pi),I\}$.
The Hamiltonian dynamics in the new variables $(\psi, I)$ is  $C^\infty$-smooth:
\[
\dot I=0, \qquad \dot\psi\vert_{I=I(h)}=\Omega(I)\vert_{I=I(h)}=\frac{\partial H_s}{\partial I}\vert_{I=I(h)}=\frac{1}{{\partial I}/{\partial H_s}\vert_{H_s=h}}=\frac{2\pi}{T_{h,s}},
\]
where $T_{h,s}$ is the period  of  a trajectory that belongs  to the cycle $\gamma_h$.  Here we use the fact that the derivative of \eqref{I} with respect to the energy $h$ is exactly the period \eqref{period} divided by $2\pi$ (see \cite{Ar}).

Now, consider the original phase space $T^*\R^2$, a singular Lagrangian toric fibration
\[
\mathbf F\colon T^*\R^{2}\to \R^2, \qquad \mathbf F=(H,S),
\]
and a regular invariant Lagrangian torus of the non-reduced Hamiltonian system
\[
\mathbb T^2_{h_0,s_0}=\mathbf F^{-1}(h_0,s_0), \qquad (h_0,s_0)\in{\mathcal Reg}.
\]

Let $U=(h_0-\varepsilon_1,h_0+\varepsilon_1)\times (s_0-\varepsilon_2,s_0+\varepsilon_2)\subset {\mathcal Reg}$. In a toroidal neighborhood $\mathcal U=\mathbf F^{-1}(U)$ of
$\mathbb T^2_{h_0,s_0}$ we take the standard action-angle variables $(I_1,I_2,\psi_1,\psi_2)$,
\begin{align*}
& I_1=I_1(p_r,p_\varphi,r_r), \quad I_2=p_\varphi, \quad \psi_j=\psi_j(p_r,p_\varphi,r_r)\, \mod\,2\pi, \quad j=1,2,\\
& \omega=dp_1\wedge dx_1+dp_2\wedge dx_2=dp_r\wedge dr+dp_\varphi\wedge d\varphi=dI_1\wedge d\psi_1+dI_2\wedge d\psi_2.
\end{align*}

 It is clear that $\mathbb T^2_{h_0,s_0} \subset \Sigma_- \cup \Sigma \cup \Sigma_+$.
 The construction is independent of the invariant decomposition of the phase space \eqref{razbijanje}. Only, for $\mathbb T^2_{h_0,s_0} \subset \Sigma_+$
 all considered objects are $C^\infty$-smooth.
The action variables are defined (see \cite{Ar}) by
\[
I_j=I_j(H,S)\vert_{\mathbb T^2_{h,s}}=\frac{1}{2\pi}\int_{\gamma_j} p_r dr + p_\varphi d\varphi, \qquad j=1,2,
\]
where $\gamma_1$ and $\gamma_2$ form a pair of basal cycles on $\mathbb T^2_{h,s}$. For example,
we can take $\gamma_1$ to be the intersection of $\mathbb T^2_{h,s}$ with the $3$-dimensional space $\{\varphi=0\,(\mod\,2\pi)\}$
and $\gamma_2$ to be one of the two components of the intersection of $\mathbb T^2_{h,s}$ with the $3$-dimensional space $\{r=r^*_{s_0}\}$, e.g. we can take those with $p_r$ coordinate greater than  zero.  There exists a natural identification of $\gamma_h$ with $\gamma_1$. We have
\begin{align*}
& I_1(H,S)\vert_{\mathbb T^2_{h,s}}=\frac{1}{\pi}\int_{r_{\min,s}}^{r_{\max,s}} \sqrt{2(h-V_s(r))}dr=I(H_s)\vert_{H_s=h}, \\
& I_2(H,S)\vert_{\mathbb T^2_{h,s}}=\frac{1}{2\pi}\int_{\gamma_2} s  d\varphi=s.
\end{align*}

We take $\psi_1\vert_{\gamma_2}=0\,(\mod\,2\pi)$ and $\psi_2\vert_{\gamma_1}=0\,(\mod\,2\pi)$.
 Let $\pi(r,\varphi,p_r,p_\varphi)=(r,\varphi)$ be the projection to the configuration space.
%Further, we cut the cycles $\gamma_1$ and $\gamma_2$ from $\mathbb T^2_{h,s}$ and denote the obtained rectangular by $\mathbb T^2_{h,s}^\Box$.
For the angle coordinates we use
the generating function
\[
\mathcal S(I_1,I_2,\pi(\mathbf b))\vert_{I_1=I_1(h,s), I_2=I_2(h,s)}:=\int_{\gamma'_{h,s}}p_r dr+p_\varphi d\varphi,
\]
where   $\mathbf b$ is an arbitrary point in $\mathbb T^2_{h,s}$, $\gamma'_{h,s}\subset \mathbb T^2_{h,s}$ is an arbitrary  curve connecting the points $ \mathbf a=\gamma_1\cap \gamma_2$ and $\mathbf b$ that is homologically trivial.
Then (see \cite{Ar})
\[
\psi_1=\psi_1(r,\varphi,p_r,p_\varphi)=\frac{\partial\mathcal S}{\partial I_1}, \qquad \psi_2=\psi_2(r,\varphi,p_r,p_\varphi)=\frac{\partial\mathcal S}{\partial I_2}.
\]

By definition, the angle $\psi_1$ is equal to the angle coordinate $\psi$, with the ray $\Gamma_s$ in the definition \eqref{ugao} replaced by the fixed ray $\Gamma_{s_0}$ for all $s\in (s_0-\varepsilon_2,s_0+\varepsilon_2)$.

Let $H=H(I_1,I_2)$ be the Hamiltonian function expressed using the action variables
and $J\colon (H,S)\mapsto (I_1,I_2)$ be the transformation defining the action variables near $\mathbb T^2_{h_0,s_0}$:
$I_1(H,S)\vert_{S=s}=I(H_s)$, $I_2=S$, Then,
\[
dJ=\begin{pmatrix}
\frac{\partial I_1}{\partial H} & \frac{\partial I_1}{\partial S}\\
0 & 1
\end{pmatrix}
\]
and the differential of the inverse map, for $I_1=I_1(h,s), I_2=s$, is given by
\[
dJ^{-1}\vert_{I_1=I_1(h,s), I_2=s}=
\begin{pmatrix}
\frac{\partial H}{\partial I_1} & \frac{\partial H}{\partial I_2}\\
 \frac{\partial S}{\partial I_1} & \frac{\partial S}{\partial I_2}
\end{pmatrix}\vert_{I_1=I_1(h,s), I_2=s}=\frac{1}{\frac{\partial I_1}{\partial H}}
\begin{pmatrix}
1 & -\frac{\partial I_1}{\partial S}\\
 0 & \frac{\partial I_1}{\partial H}
\end{pmatrix}\vert_{H=h, S=s}.
\]
Therefore,
\begin{align*}
\frac{\partial H}{\partial I_2}\vert_{I_1=I_1(h,s), I_2=s}=&
-\frac{\frac{\partial I_1}{\partial S}}{{\frac{\partial I_1}{\partial H}}}\vert_{I_1=I_1(h,s), I_2=s}=-
\frac{\frac{d}{ds}I(H_s)\vert_{H_s=h}}
{\frac{d}{dh}I(H_s)\vert_{H_s=h}}=\frac{\Delta\varphi_{h,s}}{T_{h,s}},
\end{align*}
where $T_{h,s}$ is the period \eqref{period} of a trajectory that belongs to $\gamma_h$ and $\Delta\varphi_{h,s}$ is the increase  \eqref{prirastaj} of the angle $\varphi$ of the corresponding trajectory in the original configuration space for the time $T_{h,s}$.
Thus, $\Delta\varphi_{h,s}/T_{h,s}$ is the average velocity $\dot\varphi$ along a motion with $H=h>0$, $S=s\ne 0$.

%and we get the action-angle variables $(I_1,I_2,\psi,\varphi)$ by using $C^1$-diffeomorphism

\begin{thm}\label{th:integrability}
The Hamiltonian system defined by the equations \eqref{ham1}, \eqref{ham2} is completely integrable.
The phase space $T^*\R^2$ is almost everywhere foliated on invariant tori $\mathbb T^2_{h,s}=\{H=h, S=s\}$, $(h,s)\in{\mathcal Reg}$.
In a toroidal neighborhood  of
$\mathbb T^2_{h_0,s_0}$ there exist
action-angle variables $(I_1,I_2,\psi_1,\psi_2)$ that linearize the Hamiltonian dynamics:
\begin{align*}
& \dot I_1=0,\qquad   \dot\psi_1=\Omega_1(I_1,I_2)\vert_{\mathbb T^2_{h,s}}=\frac{\partial H}{\partial I_1}\vert_{\mathbb T^2_{h,s}}=\frac{2\pi}{T_{h,s}},\\
& \dot I_2=0, \qquad  \dot\psi_2=\Omega_2(I_1,I_2)\vert_{\mathbb T^2_{h,s}}=\frac{\partial H}{\partial I_2}\vert_{\mathbb T^2_{h,s}}=\frac{\Delta\varphi_{h,s}}{T_{h,s}}.
\end{align*}
\end{thm}

Thus, the Hamiltonian system \eqref{ham1}, \eqref{ham2}  is $C^1$-conjugate to a linear flow, and it is $C^\infty$-conjugate outside \eqref{Delta}.
Note that $\Omega_i$ are smooth functions in a neighborhood  of
$\mathbb T^2_{h_0,s_0}$.
  A dynamical meaning of the angle variables is that $\dot\psi_1$ is the average velocity of the $SO(2)$-reduced system along the cycle $\gamma_h$ and $\dot\psi_2$ is the average velocity of the corresponding reconstruction problem   $\dot\varphi=s/r^2(t)$ in the angular variable $\varphi$.

\begin{rem}
Let us note that in \cite{AX}, the Hamiltonians were of the class $C^2$, while the first integrals of motion were of the class $C^1$.
\end{rem}

\subsection{$C^1$-gluing of two integrable systems}

Geometrically, the construction presented in this section can be seen as a $C^1$-gluing of two integrable systems: the free motion in $\R^2$ with noncompact invariant manifolds and the motion under the influence of the elastic potential with compact invariant manifolds over the boundary $\Delta$.
The first system is non-commutatively integrable with a complete sets of integrals
\begin{equation}\label{eq:Fi}
F_1=p_1,\quad F_2=p_2,\quad S=x_1p_2-x_2p_1 \quad (H_1=\frac12(F_1^2+F_2^2)),
\end{equation}
and the second system is Liouville integrable with the integrals
\[
H_2(x,p)=\frac12 \langle p,p\rangle+\frac12\sigma (\vert x\vert-\ell)^2, \qquad S=x_1p_2-x_2p_1.
\]

Both systems have a common integral $S$, and their Hamiltonian functions coincide at $\Delta$: $H_1\vert_\Delta=H_2\vert_\Delta$.
As a result, we have a $C^1$-gluing of the cylinder $\{H_1=h, S=s\}$ and the torus $\{H_2=h, S=s\}$, resulting with the regular invariant torus $\mathbb T^2_{h,s}$
consisting  a part of the cylinder within the region $x_1^2+x_2^2\le 1$ and a part of the torus for $x_1^2+x_2^2\ge 1$  (see Fig. 2).

\begin{figure}[h]
{\centering{\includegraphics[width=11.5cm]{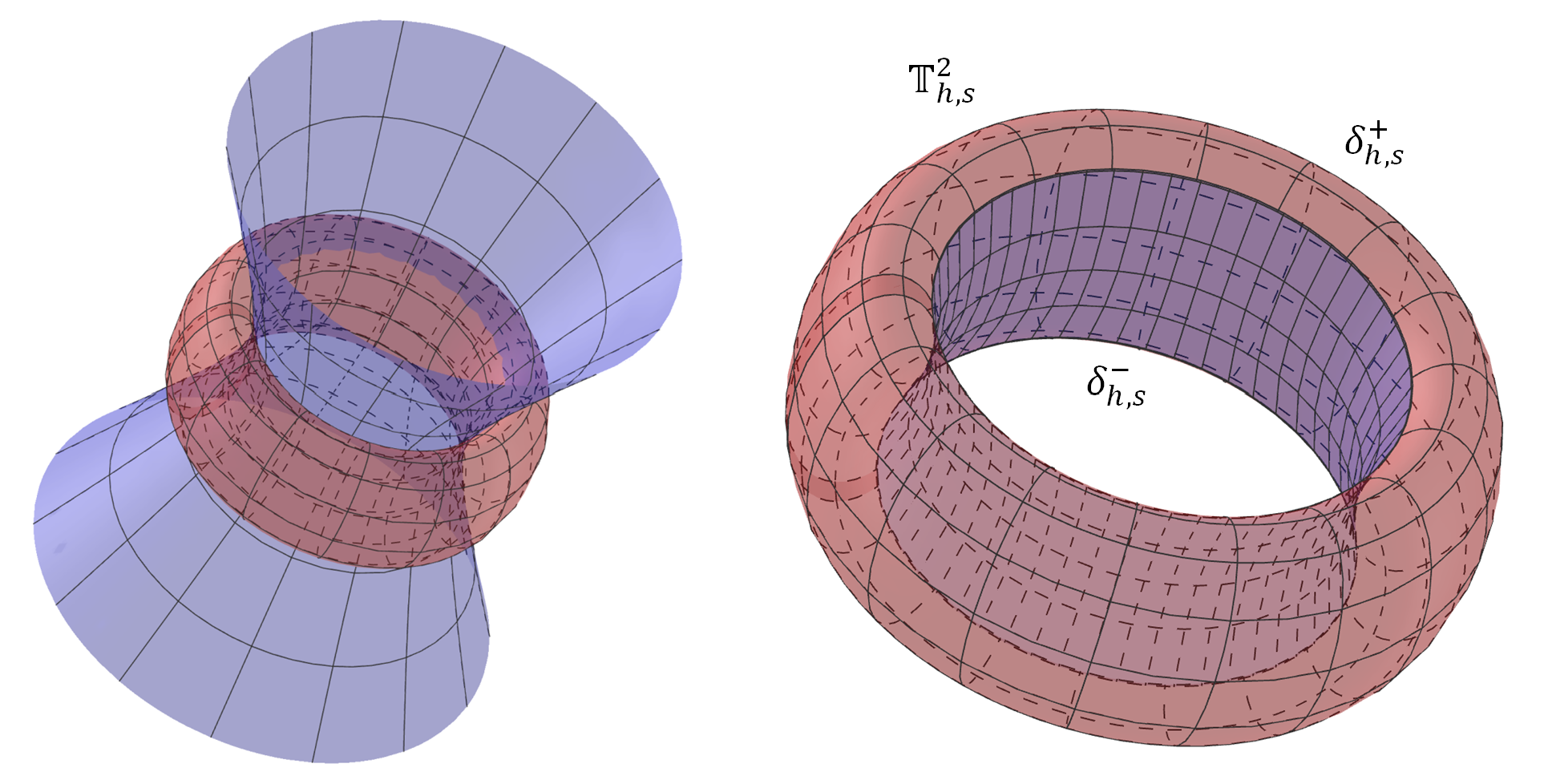}}
\caption{The cylinder $\{H_1=h, S=s\}$ and the torus $\{H_2=h, S=s\}$ in $\R^4$ (left) and the regular invariant torus $\mathbb T^2_{h,s}$
(right). The intersection  are the cycles $\delta^+_{h,s}$ and $\delta^-_{h,s}$.\label{fig:presek}}}\end{figure}

The intersection of the cylinder and the torus contains two circles $\delta^+_{h,s}$ and $\delta^-_{h,s}$, which correspond to the level sets $\{S=s\}$ of the return maps on $M_h^+$ and $M_h^-$
(see Remark \ref{S=s}). The circles $\delta^+_{h,s}$ and $\delta^-_{h,s}$ are Poincar\'e's sections on $\mathbb T^2_{h,s}$ and the rotations given by $\Theta^h_+$ and $\Theta^h_-$ are the associated Poincar\'e's maps. This is  schematically illustrated in Figure  \ref{fig:lepljenje} for $C^0$-approximation of the system.

\section{$C^0$-approximative system}\label{sec3}

\subsection{Definition of the $C^0$-approximative system} For $g> 0$ and $\vert x\vert >\ell$, the Hamiltonian system defined by \eqref{eq:ham*} is an elastic pendulum, which is not integrable
(see e.g. \cite{Ku, L, MP, AMBC} and references therein).
Therefore, we can not hope to determine the corresponding map $\Psi$ (see equation \eqref{Psi}) by quadratures.

\begin{rem}
Note that in the region $\Sigma_+$ for a sufficiently small $g$, we can apply the $C^\infty$-smooth KAM theorem (see e.g., \cite{Ar, TZ}) in order to conclude that some of the invariant non-resonant Lagrangian tori are preserved.
 In contrast to that, an interesting problem is to investigate the preservation of the invariant tori after a perturbation with the gravity potential $g x_2$ within the region $\Sigma_-$,  where the invariant tori are  only $C^1$-smooth at $\Delta$.
\end{rem}

It is well known that some qualitative properties of the elastic pendulum in dimension $3$ were obtained from its integrable approximation (see \cite{DGC}).

In our case, in order to have an explicit description for the mappings $\Psi$ in the presence of gravity, instead of the
potential
\[
V^*(x)=gx_2+\frac12\sigma (\vert x\vert-\ell)^2\theta(\vert x\vert-\ell),
\]
we will consider its following approximation
\[
V(x)=gx_2+\frac12\rho (\vert x\vert^2-{\ell^2})\theta(\vert x\vert-\ell),
\]
where $\rho>0$ is a real parameter.
 For $\vert x\vert\ge\ell$, $V(x)$ is the  Hook potential of the elastic force centered at  $C(0,-g/\rho)$:
\begin{equation}\label{elasticni}
V(x)=\frac12\rho\Big(x_1^2+\Big(x_2+\frac{g}{\rho}\Big)^2\Big)-\frac{g^2}{2\rho}-\frac{\rho\ell^2}2.
\end{equation}

Now, the new Hamiltonian \eqref{eq:ham2}
is $C^0$-continuous at $\Delta$ and $C^{\infty}$-smooth  outside $\Delta$.
Therefore, the Hamiltonian vector field {\it is not continuous} at $\Delta$.  Within the disk $D$  (see \eqref{eq:disk}), the Hamiltonian equations of motion are
\begin{equation}\label{ham1g}
\dot x=p, \qquad \dot p=-g\mathbf e_2,
\end{equation}
while outside the disk $D$, the Hamiltonian system takes the form
\begin{equation}\label{ham2g}
\dot x=p, \quad \dot p=-g\mathbf e_2+\mathbf G, \qquad \mathbf G=
-\rho x \quad \big(\mathbf G\vert_{\partial D}\neq 0\big).
\end{equation}
The trajectories of \eqref{ham1g} are parabolas or vertical rays:
\begin{equation}\label{realnaunutar}
x_1(t)=x_{10}+p_{10}t,\ \ x_2(t)=x_{20}+p_{20}t-\frac{1}{2}gt^2.
\end{equation}
The trajectories of \eqref{ham2g} are ellipses or segments, centered at $C(0,-g/\rho)$:
\begin{equation} \label{realnavan}
\begin{aligned}
x_1(t)&=x_{10}\cos\sqrt{\rho}t+\frac{p_{10}}{\sqrt{\rho}}\sin\sqrt{\rho}t,\\ x_2(t)&=\Big(x_{20}+\frac{g}{\rho}\Big)\cos\sqrt{\rho}t+\frac{p_{20}}{\sqrt{\rho}}\sin\sqrt{\rho}t-\frac{g}{\rho}.
\end{aligned}
\end{equation}

It is convenient to rewrite \eqref{realnaunutar} and \eqref{realnavan} using a complex notation:
\begin{align}\label{unutar}
\mathbf x(t)&=x_1(t)+\i x_2(t)=\mathbf x_0+t \mathbf p_0-\frac12 g t^2 \i, \qquad \mathbf x_0, \mathbf p_0\in\mathbb C,\\
%x(t)=x_0+tp_0-\frac12 g t^2\mathbf e_2=(x_{1,0}+p_{1,0}t,x_{2,0}+p_{2,0}t-\frac12 gt^2),
\label{van}
\mathbf x(t)&=x_1(t)+\i x_{2}(t)=\mathbf y_1e^{\i\sqrt{\rho}t}+\mathbf y_2e^{-\i\sqrt{\rho}t}-\frac{g}{\rho}\i, \qquad \mathbf y_0,\mathbf y_1\in\mathbb C.
\end{align}

The Hamiltonian dynamics near  $\Delta\setminus \Delta_0=\Delta_+\cup\Delta_-$ is defined as follows  (see \eqref{Delta}--\eqref{Delta-}).

\begin{itemize}
\item[\textbf{R0}] Let $(x(0),p(0))\in\Delta\setminus\Delta_0$. If $p(0)$ is directed outside the disk $D$ then for $t\in (-\varepsilon,0]$ the Hamiltonian motion is given by \eqref{van}, and for $t\in [0,\varepsilon)$, the motion is given by \eqref{unutar}. Vice versa, if $p(0)$ is directed inside, then for $t\in (-\varepsilon,0]$ the motion is given by \eqref{unutar}, and for $t\in [0,\varepsilon)$, the motion is given by \eqref{van}.
Let $\mathbf p=p_1+\i p_2$.
The values of $\mathbf x_0, \mathbf p_0, \mathbf y_1, \mathbf y_2$ are uniquely determined from the equations
\begin{align}\label{pocetni}
\mathbf x(0)=\mathbf x_0=\mathbf y_1+\mathbf y_2-\frac{g}{\rho}\i,
\qquad \mathbf p(0)=\mathbf p_0=\i\sqrt{\rho}(\mathbf y_1-\mathbf y_2).
\end{align}
\end{itemize}

Then the solution $\mathbf x(t)$ is $C^1$-smooth at $0$, while $\mathbf p(t)$ is $C^0$-continuous at $0$.
The Hamiltonian function is preserved along the motion. Note that the value of the Hamiltonian function in terms of the initial data
\eqref{pocetni} is given by
\begin{equation}\label{poc:en}
H=\frac12\mathbf p_0\bar{\mathbf p}_0+\Im(\mathbf x_0)g=\frac12 \rho(\mathbf y_1-\mathbf y_2)(\bar{\mathbf y}_2-\bar{\mathbf y}_1)
+\Im\Big(\mathbf y_1+\mathbf y_2-\frac{g}{\rho}\i\Big)g.
\end{equation}

Further, we use the following two rules to extend the dynamics:
\begin{itemize}
\item[\textbf{R1}] If a solution $\mathbf x(t)$, $t\in(-\varepsilon,\varepsilon)$ of \eqref{ham2g}, is such that
$\vert\mathbf x(t)\vert\ge\ell$ that is tangent to the circle $\partial D$ at $0$, including the case when the velocity is equal to zero at $0$,
then it is a smooth solution of the Hamiltonian system \eqref{eq:ham2}.

\item[\textbf{R2}]  If we have a solution
$\mathbf x(t)$, $t\in(-\varepsilon,\varepsilon)$ of \eqref{ham1g}, such that
$\vert\mathbf x(t)\vert\le\ell$ that is tangent to the circle $\partial D$ at $0$, including the case when the velocity is equal to zero at $0$, then it is a smooth solution of the Hamiltonian system \eqref{eq:ham2}.
\end{itemize}

There exists a set $\Pi\subset \Delta_0$ of initial conditions that are not covered with the above cases \textbf{R1} and \textbf{R2} and for these initial conditions we assume that the dynamics is not defined.
For example, the point $(x(0),p(0))=(\pm \ell,0,0,0)$ belongs to $\Pi$.
However, all the  initial conditions within $T^*\R^2\setminus \Delta_0$ lead to well defined trajectories.
\begin{figure}[h]
\begin{center}
{\includegraphics[width=10cm]{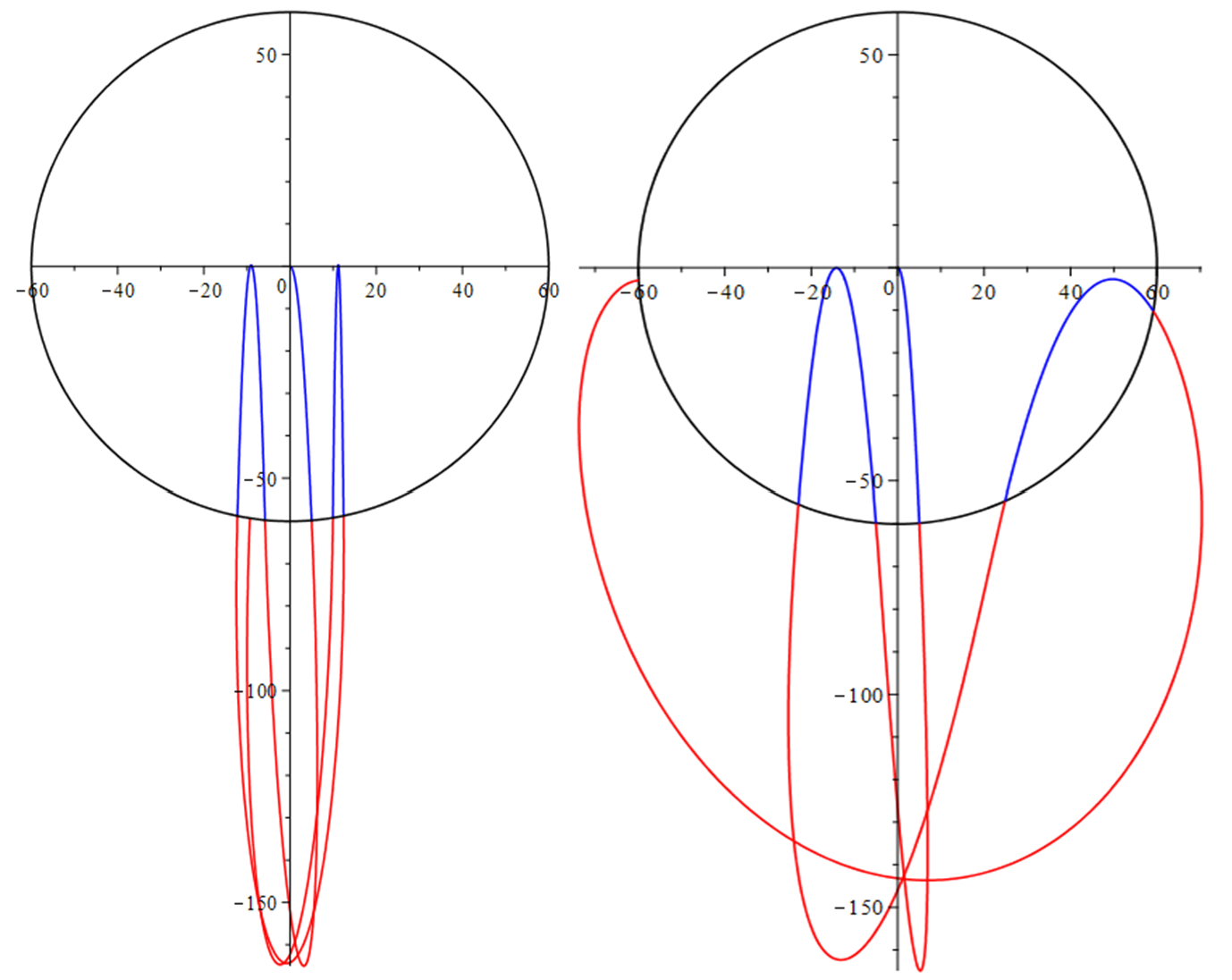}
\caption{The trajectory of the $C^0$-model  are on the left and of the $C^1$-model on the right of bungee jumping at the Verzasca Dam that starts at the origin, with the initial velocity $p_{10}=1.43m/s$, $p_{20}=0$.
Inside the disk $D$, the trajectories  in both cases are arcs of parabolas, represented in blue. Outside the disk, the trajectories  are  arcs of ellipses on the left, represented in red, while on the right are the trajectories of the  elastic pendulum.}
\label{fig:svica}}
\end{center}
\end{figure}
By the \emph{free jump} we mean a motion with the initial condition
\[
(x_1(0),x_2(0),p_1(0),p_2(0))=(0,0,0,0).
\]

Let $R^*_{\max}, R_{\max}$ be the maximal distance from the origin of the free jump of the original
system defined by \eqref{eq:ham*} and the approximative system defined by \eqref{ham1g} and \eqref{ham2g}, respectively:
\begin{align}
\label{rM*} &\sigma {R^{*}_{\max}}^2- 2(\sigma \ell+g) R^*_{\max}+  \sigma {\ell^2}=0,\\
\label{rM}  &\rho {R_{\max}}^2-2g R_{\max}-\rho  {\ell^2}=0.
\end{align}

From \eqref{rM*} and \eqref{rM}, we get
\begin{align*}
& R^{*}_{\max}=l+\frac{g}{\sigma} +\sqrt{\Big(\ell+\frac{g}{\sigma}\Big)^2-{\ell^2}}=\ell+\frac{g}{\sigma}+\sqrt{\frac{g^2}{\sigma^2}+2\frac{\ell g}{\sigma}}, \\
& R_{\max}=\frac{g}{\rho} +\sqrt{\frac{g^2}{\rho^2}+{\ell^2}}
\end{align*}

In addition, we consider the equilibrium position  $C^*$
of the original system (defined by \eqref{eq:ham*}) and,
for $g/\rho>\ell$, the equilibrium position $C$ of the approximative system:
\begin{equation*}
 C^*\big(0,-(\ell+\frac{g}{\sigma})\big), \qquad C\big(0,-\frac{g}{\rho}\big).
\end{equation*}

A natural choice for the parameter $\rho$ is given by the conditions: $R_{\max}^*=R_{\max}$ and $C^*=C$.
However, these equations are overdetermined and without a solution.
We consider the following two cases:

\begin{itemize}
\item{} If we are interested in trajectories with initial conditions  $(x(0), p(0))$ close to $(0,0,0,0)$, then
we assume $R_{\max}=R^*_{\max}$ and
instead of the condition $C^*=C$, a weaker condition  $g/\rho>\ell$ (the existence of the
equilibrium position $C$) is assumed.

\item{} If we are interested in trajectories with initial conditions  $(x(0),p(0))$ close to  $C^*$, then
we assume $C=C^*$, that is,
\[
\rho=\frac{\sigma g}{\sigma \ell+g}.
\]
%Note that in the second case, the dynamics is well defined even when $(x(0),p(0))$ belongs to $\Delta_0$.
\end{itemize}

The above analysis, in general, suggests that it is natural to assume that the point $C(0,-g/\rho)$ is outside the disk $D$, which leads to the assumption:
\begin{equation}\label{uslov:rho}
\frac{g}{\rho}>\ell.
\end{equation}

The point $C$ is the equilibrium point of the approximative system \eqref{ham1g} and \eqref{ham2g}  and this point is the center of all the ellipses and segments given by \eqref{van}.

\begin{exm}
Consider the bungee jumping at the Verzasca Dam located in the Ticino region of
Switzerland.  The maximal possible stretched length of the elastic cord is about 220m.\footnote{https://www.007bungy.ch/} We model the free jump
by  the cord of the length $\ell=60 m$ that is stretched to $R^*_{\max}=165 m$. The trajectories of our model with the unit mass correspond to the
trajectories of the body with the mass $m$ by denoting $\sigma/m$ and $\rho/m$ by $\sigma$ and $\rho$, respectively. From the condition $R_{\max}=R^*_{\max}=165m$, we get
\begin{align*}
&165 m=60 m+\frac{g}{\sigma} +\sqrt{\frac{g^2}{\sigma^2}+120m\frac{g}{\sigma}}  \Rightarrow    \vert OC^*\vert=60 m+\frac{g}{\sigma}=60 m+\frac{11025}{330}m\approx 93,41m,\\
&165 m=\frac{g}{\rho} +\sqrt{\frac{g^2}{\rho^2}+(60 m)^2} \quad   \Rightarrow \quad \vert OC\vert=\frac{g}{\rho}=\frac{23625}{330}m\approx 71,59m> 60m.
\end{align*}
The trajectory that starts at the origin, with the initial velocity given by $p_{10}=1.43 m/s$, $p_{20}=0$ with  five crossings through $\partial D$ and ending at $\partial D$ is presented in the Figure \ref{fig:svica},  in two ways: using  the $C^0$-model (left) and the $C^1$-model (right).
The value of the gravitational constant $g$ is $9.81m/s^2$.

\end{exm}

\subsection{The return maps}

%In the case of $C^0$-approximation one can find  the billiard mappings $\Theta_{\pm}\colon \Delta_\pm\to\Delta_\pm$ in the explicit algebraic forms.
Let  $\Psi$, $\Phi$,  $\Theta_{+}=\Phi\circ\Psi$, and $\Theta_{-}=\Psi\circ\Phi$
be defined as in Section \ref{sec2}.   The maps $\Psi$ and $\Phi$ correspond to two successive  intersections  with the boundary circle $\partial D$ of the trajectories outside the disc $D$ and within the disk $D$, respectively. A single trajectory of the system defines a sequence of points on $\Delta_+$ and $\Delta_-$, the trajectories of the return maps $\Theta_+$ and $\Theta_-$ (see Figure \ref{fig:lepa} (left)).

\begin{figure}[h]
\begin{center}
{\includegraphics[width=9cm]{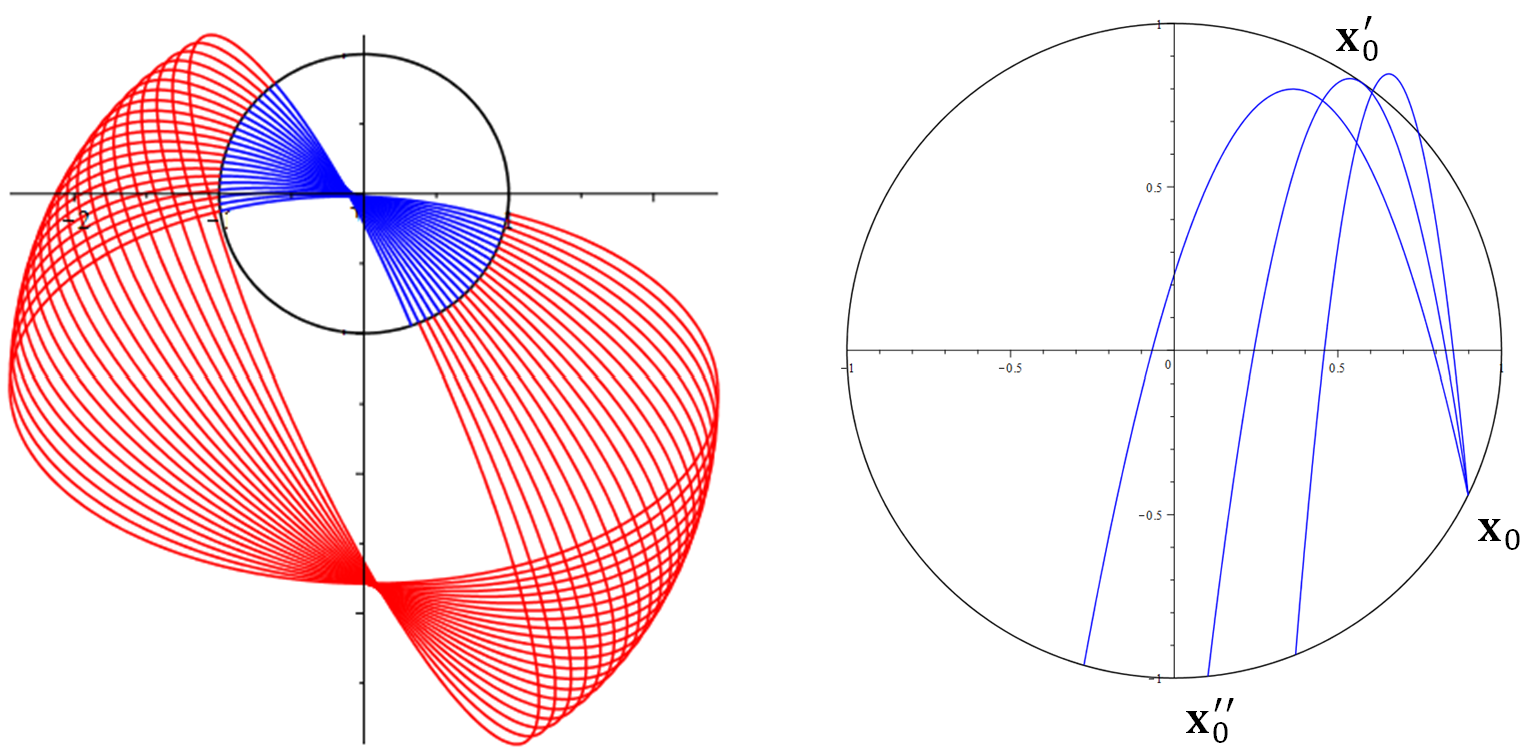}
\caption{On the left: a trajectory of $C^0$-approximate system with a positive energy; on the right: three parabolas with nearby initial conditions staring at $\mathbf x_0$ and having the same positive energy, one of which is tangent to the boundary circle. \label{fig:lepa}}
}
\end{center}
\end{figure}

%The return map $\Theta_{+}=\Phi\circ\Psi$ and $\Theta_{-}=\Psi\circ\Phi$ will be calculated in an explicit algebraic form.
However, the mappings $\Phi$ and $\Psi$ are not continuous. Indeed, let
\[
\mathcal P(\mathbf{x}_0,\mathbf p_0)=\Big\{\mathbf x(t)=\mathbf x_0+t \mathbf p_0-\frac12 g t^2 \i\,\vert\, t\in\mathbb R\Big\}
\]
be the parabola or a vertical ray, with the initial condition  $(\mathbf{x}_0,\mathbf p_0)\in\Delta_{-}$. Consider the case when $\mathcal P(\mathbf{x}_0,\mathbf p_0)$ is  parabola that is tangent to $\partial D$ at the point $\mathbf{x}_0'$ and intersects the boundary circle  $\partial D$ at $\mathbf{x}_0''$. Then, $\mathbf{x}_0''$ is the $\mathbf{x}$-component of $\Phi(\mathbf{x}_0,\mathbf p_0)$.
 By varying  of the initial condition $\mathbf p_0$, one gets the $\mathbf{x}$-component of $\Phi(\mathbf{x}_0,\mathbf p_0)$  close to two distinct points $\mathbf{x}_0'$ and $\mathbf{x}_0''$ (see Figure \ref{fig:lepa} (right)). Discontinuity appears also when $\mathcal P(\mathbf{x}_0,\mathbf p_0)$ is a vertical ray that touches the boundary circle $\partial D$ at  $\mathbf{x}_0'$.  Similarly, there is a discontinuity in the map $\Psi$, when the outgoing ellipse is tangent to $\partial D$.

Our aim is to construct $C^\infty$-smooth return maps for a fixed energy level and for all the initial conditions.
 We restrict our study to the portion of the phase space with the negative energy  $H(\mathbf{x}_0,\mathbf p_0)=h<0$, where
$(\mathbf{x}_0,\mathbf p_0)\in \Delta$. Then,
 $x_2<0$ and there are no parabolas that are tangent to $\partial D$ and there are no rays that touch $\partial D$.

Under this assumption, we first analyse and describe the map $\Psi$. Assume $(x(0),p(0))\in \Delta_+$ and consider the trajectory $\mathbf x(t)$ given by \eqref{van},  \eqref{pocetni},
and the corresponding ellipse or the segment
\[
\mathcal E(\mathbf x_0,\mathbf p_0)=\Big\{\mathbf x(t)=\mathbf y_1e^{\i\sqrt{\rho}t}+\mathbf y_2e^{-\i\sqrt{\rho}t}-\frac{g}{\rho}\i\,\vert\, t\in\mathbb R\Big\},
\]
centered at $C$.
Note that, due to the form of the potential \eqref{elasticni} and the negative energy condition $H=h<0$,  $\mathcal E(\mathbf x_0,\mathbf p_0)$ belongs to the disk $D_1$ centered at $C$ of radius $\sqrt{(g/\rho)^2+\ell^2}$  (see Figure \ref{fig:krugovi} (left)):
\begin{equation}\label{disk1}
D_1=\Big\{\mathbf x\in\mathbb C\,\vert\, \big(\mathbf x+\i\frac{g}{\rho}\big)\big(\bar{\mathbf x}-\i\frac{g}{\rho}\big)<\frac{g^2}{\rho^2}+\ell^2\Big\}.
\end{equation}

The condition \eqref{uslov:rho} is satisfied and the point $C$ is outside the disk $D$.
Let $M,N\in\partial D$ be the points  of the intersections of the tangent lines through $C$ to the circle $\partial D$.   If we consider the initial condition $(\mathbf x_0,\mathbf p_0)\in\Delta_+$ with the energy $H=h<V(M)=V(N)$,  then $\mathcal E(\mathbf x_0,\mathbf p_0)$  belongs to the disk $D_2$ with the center $C$ and radius $\vert CM\vert=\vert CN\vert$. Consequently, $\mathcal E(\mathbf x_0,\mathbf p_0)$ has only two points of intersection in the case of the ellipse, or only one point of intersection in the case of the segment, with the boundary circle $\partial D$ (Figure \ref{fig:krugovi} (right)).  Thus, we avoid the cases when the ellipse is tangent to $\partial D$ and has three  geometrically distinct intersection points, although it belongs to $D_1$ (Figure \ref{fig:krugovi} (left)).

\begin{figure}[h]
\begin{center}
{\includegraphics[width=12cm]{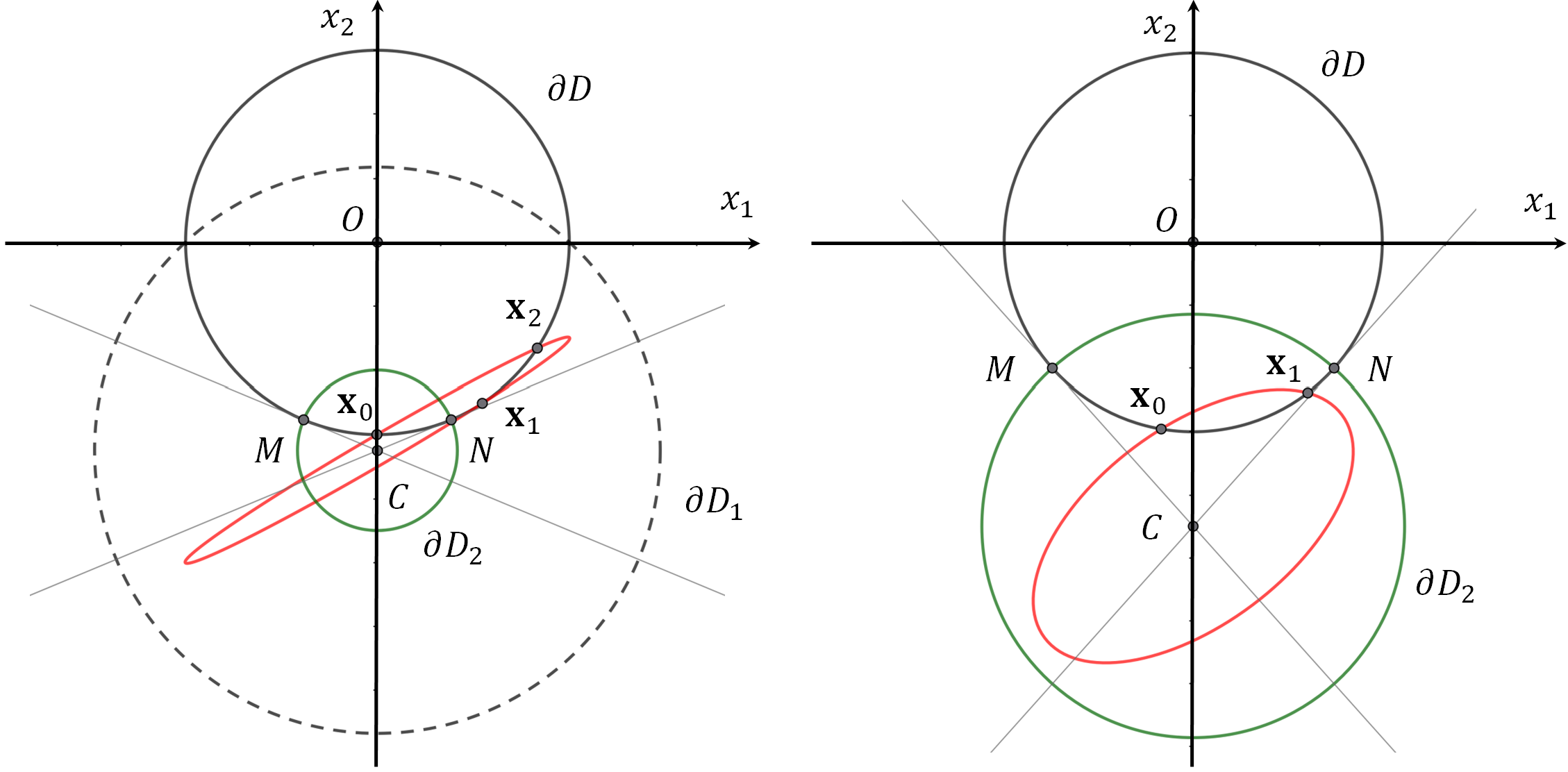}
\caption{ An ellipse $\mathcal E(\mathbf x_0,\mathbf p_0)$ that belongs to $D_1$, which is tangent to the boundary circle $\partial D$ (left). The corresponding solution
$(\mathbf x(t),\mathbf p(t))$ given by \eqref{van} intersects $\Delta$ in the points $(\mathbf x_0,\mathbf p_0)\in\Delta_+$, $(\mathbf x_1,\mathbf p_1)\in\Delta_0$, and $(\mathbf x_2,\mathbf p_2)\in\Delta_-$. All ellipses $\mathcal E(\mathbf x_0,\mathbf p_0)$, $(\mathbf x_0,\mathbf p_0)\in\Delta_+$, that are within $D_2$, intersect $\partial D$ in exactly two points (right).  \label{fig:krugovi}}
}
\end{center}
\end{figure}

 Now, let $(\mathbf x_0,\mathbf p_0)\in\Delta_+$ be an arbitrary initial condition.
We want to find the first moment of time $\tau>0$ such that
\[
\mathbf x(\tau)\bar{\mathbf x}(\tau)={\ell^2}.
\]

\begin{lem}\label{lem3}
Denote $e^{\i\sqrt{\rho}\tau}$ by $z$; $z$ is a solution of
 the complex reciprocal polynomial equation of the fourth degree
\begin{equation}
a z^4+b z^3-(a+\bar a+b+\bar b)z^2+\bar b z+\bar a=0,
\label{polinom4}
\end{equation}
where
$$\begin{aligned}
a&=a_1+\i a_2=\mathbf y_1\bar{\mathbf y}_2=\frac{1}{4}\Big(\mathbf{x}_0+\i\frac{g}{\rho}-\i\frac{\mathbf{p}_0}{\sqrt{\rho}}\Big)\Big(\bar{\mathbf{x}}_0-\i\frac{g}{\rho}-
\i\frac{\bar{\mathbf{p}}_0}{\sqrt{\rho}}\Big)\\
b&=b_1+\i b_2=\i\frac{g}{\rho}(\mathbf y_1-\bar{\mathbf y}_2)=\frac{\i g}{2\rho}\Big(\mathbf{x}_0-\bar{\mathbf x}_0+\frac{2\i g}{\rho}-\frac{\i}{\sqrt{\rho}}(\mathbf{p}_0-\bar{\mathbf{p}}_0)\Big).
\end{aligned}
$$
The parameter $\alpha\in\mathbb R$ defined by $e^{\i\sqrt{\rho}\tau}=\cos\alpha+\i\sin\alpha$ satisfies the trigonometric equation:
\begin{equation}
\sin\frac{\alpha}{2}\Big(2a_1\sin\alpha\cos\frac{\alpha}{2}+2a_2\cos\alpha\cos\frac{\alpha}{2}+b_1\sin\frac{\alpha}{2}+b_2\cos\frac{\alpha}{2}\Big)=0.
\label{trigonometrijska}
\end{equation}
\end{lem}

\begin{proof}
By using that $e^{-\i\sqrt{\rho}\tau}=z^{-1}=\bar{z}$,  one gets that $z$ is a solution of the equation:
\[
\Big(\mathbf y_1 z+\mathbf y_2 z^{-1}-\frac{g}{\rho}\i\Big)\Big(\bar{\mathbf y}_1 z^{-1}+\bar{\mathbf y}_2 z+\frac{g}{\rho}\i\Big)=\Big(\mathbf y_1 +\mathbf y_2-\frac{g}{\rho}\i\Big)
\Big(\bar{\mathbf y}_1 +\bar{\mathbf y}_2+\frac{g}{\rho}\i\Big),
\]
which is equivalent to \eqref{polinom4}.
The above  equation in terms of $\alpha$ implies
$$
\begin{aligned}
a (\cos2\alpha+\i\sin2\alpha)&+b (\cos\alpha+\i\sin\alpha)\\
&+\bar b(\cos\alpha-\i\sin\alpha) +\bar a(\cos2\alpha-\i\sin2\alpha)=a+\bar a+b+\bar b.
\end{aligned}
$$
After applying standard trigonometric identities, the last equation reduces to \eqref{trigonometrijska}.
\end{proof}

The equation \eqref{polinom4} can have two or four solutions of modulus $1$, counting multiplicity.

In the case of the segment, regardless of
the number of the solutions of modulus $1$, we have $\Psi(\mathbf x_0,\mathbf p_0)=(\mathbf x_0,-\mathbf p_0)$.

In the case of the ellipse, the solutions  $z$ of modulus $1$ correspond to the distinct points of the intersections of the circle $\partial D$ and the ellipse $\mathcal E(\mathbf x_0,\mathbf p_0)$.
 For a general initial conditions $(\mathbf x_0,\mathbf p_0)\in\Delta_+$, we need a careful analysis to find the first intersection point $\mathbf x(\tau)$ after $\mathbf x_0$ along the solution $\mathbf x(t)$.
A discussion with four intersection points for $g=0$ will be given in the next Section.

The $x_2$-coordinate
of the points $M$ and $N$ is $-\ell^2/\vert OC\vert=-\ell^2\rho/g$. Therefore, for (see \eqref{poc:en})
\begin{equation}\label{uslov:energija}
H=\frac12\mathbf p_0\bar{\mathbf p}_0+\Im(\mathbf x_0)g=\frac12 \rho(\mathbf y_1-\mathbf y_2)(\bar{\mathbf y}_2-\bar{\mathbf y}_1)
+\Im\Big(\mathbf y_1+\mathbf y_2-\frac{g}{\rho}\i\Big)g=h<-\rho\ell^2,
\end{equation}
the ellipse, or the segment, $\mathcal E(\mathbf x_0,\mathbf p_0)$  belongs to the disk
\[
D_2=\Big\{\mathbf x\in\mathbb C\,\vert\, \big(\mathbf x+\i\frac{g}{\rho}\big)\big(\bar{\mathbf x}-\i\frac{g}{\rho}\big)<\frac{g^2}{\rho^2}-\ell^2\Big\}.
\]
 Then, there are only two solutions  $z$ of the polynomial equation \eqref{polinom4} of modulus $1$, that is,
only two solutions $\alpha$ of the trigonometric equation \eqref{trigonometrijska} in the interval $[0,2\pi)$. Moreover, the map $\Psi$ is $C^\infty$-smooth at
$(\mathbf x_0,\mathbf p_0)$.

We proceed with the trigonometric equation \eqref{trigonometrijska}. We are looking for its solutions $\alpha$ in the interval $(0,2\pi)$.
The value $\alpha=0$ corresponds to $z=1$, i.e., to the initial point.
Recall that $a_1, a_2, b_1, b_2$ are defined in Lemma \ref{lem3}.

\begin{thm}
\label{th:prva}
Suppose that $(\mathbf{x}_0,\mathbf{p}_0)\in\Delta_+$ satisfies the energy condition \eqref{uslov:energija}. The map $
\Psi:(\mathbf{x_0}, \mathbf{p_0})\mapsto(\mathbf{x_1}, \mathbf{p_1})
$ is given by
\begin{align}
\mathbf{x}_1&=x_{10}\cos{\alpha}+\frac{p_{10}}{\sqrt{\rho}}\sin\alpha+\i\Big(\big(x_{20}+\frac{g}{\rho}\big)\cos\alpha+\frac{p_{20}}{\sqrt{\rho}}\sin\alpha-\frac{g}{\rho}\Big),\\
\mathbf{p}_1&={p_{10}}\cos\alpha-\sqrt{\rho}x_{10}\sin{\alpha}+\i\Big({p_{20}}\cos\alpha-\sqrt{\rho}\big(x_{20}+\frac{g}{\rho}\big)\sin\alpha\Big),
\end{align}
 where $\tan\frac{\alpha}{2}=u_1$, and
\begin{align}
\label{uu1}&u_1=\sqrt[3]{w_1}+\sqrt[3]{w_2}+\frac{2a_2-b_2}{3b_1},\\
\label{D1e}&w_{1,2}=\frac{-q_e}{2}\pm \frac{\sqrt{-\mathcal D_1}}{6\sqrt{3}b_1^2},\\
&\mathcal D_1=\frac{4}{27b_1^2}\Big(b_1^2+3(2a_2-b_2)(4a_1+b_1)\Big)^3-108b_1^4(q_e)^2,\\
\nonumber &q_e=\frac{1}{54b_1^3}\Big(2(-2a_2+b_2)^3-9b_1(-2a_2+b_2)(4a_1+b_1)-27b_1^2(2a_2+b_2)\Big).
\end{align}
\end{thm}
\begin{proof} One solution of the equation \eqref{trigonometrijska} is  $\alpha=0$. This corresponds to the initial point. Since the center of $\mathcal E(\mathbf x_0,\mathbf p_0)$ is outside the disk $D$, the antipodal points of the circle $\partial D$ cannot be obtained as intersection points with $\mathcal E(\mathbf x_0,\mathbf p_0)$. Hence,  $\alpha=\pi$ is not a  solution of \eqref{trigonometrijska}. This fact can be obtained also directly from the equation \eqref{trigonometrijska}. An angle $\alpha$ that solves  $\cos\frac{\alpha}{2}=0$ is a solution of \eqref{trigonometrijska} only when $b_1=0$. But $b_1\ne 0$ because
\[
b_1=\frac{g}{4\rho}x_{20}-\frac{g^2}{\rho^2}<0.
\]
Thus, the equation \eqref{trigonometrijska} can be divided by $\cos\frac{\alpha}{2}$.

Set $u=\tan\frac{\alpha}{2}$. The equation \eqref{trigonometrijska} reduces to the cubic polynomial equation
\begin{equation}\label{polinom33}
P_3(u):=b_1u^3+u^2(b_2-2a_2)+u(4a_1+b_1)+2a_2+b_2=0.
\end{equation}

 The discriminant of the cubic polynomial $P_3$  is
 \[
 \mathcal{D}(P_3)=\mathcal D_1,
 \]
where $\mathcal D_1$ is given by \eqref{D1e}.  Since real solutions of \eqref{polinom33} correspond to the intersection points of
$\mathcal E(\mathbf x_0,\mathbf p_0)$ with the boundary circle $\partial D$, due to the energy condition \eqref{uslov:energija},
the discriminant $\mathcal{D}(P_3)$ is negative. Thus, the equation \eqref{polinom33} has one real and two complex conjugated solutions. The real solution is given by \eqref{uu1}. The complex conjugate solutions  are
$$
\begin{aligned}
u_2&=\bar{\varepsilon}_1\sqrt[3]{w_1}+\varepsilon_1\sqrt[3]{w_2}+\frac{2a_2-b_2}{3b_1}, \qquad u_3=\bar{u}_2,
\end{aligned}
$$
where $\varepsilon_1=(-1+\i\sqrt{3})/2$ and $w_{1,2}$ are given by \eqref{D1e}. The real solution determines the map $\Psi$.
\end{proof}

Next, consider the case when the motion is in the interior of the disk $D$. Consider the trajectory with  the initial condition $(\mathbf{x}(0),\mathbf{p}(0))\in\Delta_-$ and the negative energy $H(\mathbf{x}(0),\mathbf{p}(0))<0$.
 Then, as it has been explained  above, $\mathcal P(\mathbf{x}_0,\mathbf p_0)$ and the boundary circle  $\partial D$ have exactly two intersection points  in the case of the parabola, or one intersection point in the case of the vertical ray. So, the map $\Phi$ is  $C^{\infty}$-smooth  at $(\mathbf{x}_0,\mathbf p_0)$.

\begin{thm}\label{th:druga}
Suppose that $(\mathbf{x}_0,\mathbf{p}_0)\in \Delta_-$ satisfies the negative energy condition $H=h<0$. The map $
\Phi:(\mathbf{x_0}, \mathbf{p_0})\mapsto(\mathbf{x_1}, \mathbf{p_1})
$ is given by
\begin{align}\label{eq:Fi}
\mathbf{x_1}&=x_{10}+p_{10}\tau+\i\Big(x_{20}+p_{20}\tau-\frac{1}{2}g\tau^2\Big),\\
\mathbf{p_1}&=p_{10}+\i(p_{20}-g\tau), \qquad \tau=\sqrt[3]{v_1}+\sqrt[3]{v_2}+\frac{4p_{20}}{3g},
\end{align}
 where
\begin{align}
\label{wp12} & v_{1,2}=-\frac{q_p}{2}\pm\frac{8\sqrt{-\mathcal D_2}}{3\sqrt{3}g^4}\\
\label{D1}&\mathcal D_2=\frac{64}{27g^4}\Big(g^2p_{20}^2-\frac{3}{4}g^2A\Big)^3-\frac{27g^8q_p}{64},\\
&q_p=\frac{32}{27g^6}\Big(-2g^3p_{20}^3+\frac{9}{4}Ag^3p_{20}+\frac{27}{16}g^4B\Big),\\
\label{ab}& A=p_{10}^2+p_{20}^2-x_{20}g, \quad  B=2x_{10}p_{10}+2x_{20}p_{20}.
\end{align}
\end{thm}
\begin{proof}
Since both points $\mathbf{x}_0$ and $\mathbf{x}_1$ belong to $\partial D$, the time $\tau$ that determines  the intersection point of the parabola and the boundary circle is a solution of the equation
\begin{equation}\label{cubic}
(x_{10}+p_{10}\tau)^2+\Big(x_{20}+p_{20}\tau-\frac12g\tau^2\Big)^2=x_{10}^2+x_{20}^2=\ell^2.
\end{equation}
One solution $\tau=0$ corresponds to the initial point. The equation \eqref{cubic} reduces to
$$
Q_3(\tau):=\frac{g^2}{4}\tau^3-gp_{20}\tau^3+A\tau+B=0,
$$
where $A$ and $B$ are given by \eqref{ab}.
The discriminant of this cubic polynomial is
\[
\mathcal{D}(Q_3)=-\frac{g^2}{16}\mathcal D_2,
\]
with $\mathcal D_2$ given by \eqref{D1}.

Since the energy in negative,  $\mathcal P(\mathbf{x}_0,\mathbf p_0)$ and the boundary circle  $\partial D$ have exactly two intersection points in the case of the parabola, or one intersection point in the case of the vertical ray.   The complex conjugate solutions  are
\[
\tau_1=\bar{\varepsilon}_1\sqrt[3]{v_1}+\varepsilon_1\sqrt[3]{v_2}+\frac{4p_{20}}{3g}, \qquad \tau_2=\bar{\tau}_1,
\]
where $\varepsilon_1=(-1+\i\sqrt{3})/2$  and $v_{1,2}$ are given by \eqref{wp12}.
\end{proof}

Consider the isoenergetic level sets
\begin{equation*}\label{fiksirana}
M_{h}^\pm=\Big\{H=\frac12\langle p,p\rangle+gx_2=h\Big\}\cap \Delta_\pm.
\end{equation*}

Since the energy is conserved,  for $h<-\rho\ell^2$, the  restrictions of the
return maps are well defined
\[
\Theta_+^h=\Phi\circ\Psi\vert_{M_h^+}\colon M_{h}^+ \to M_{h}^+, \quad \Theta_-^h=\Psi\circ\Phi\vert_{M_h^-}\colon M_{h}^-\to M_{h}^-,
\]
and they are $C^{\infty}$-smooth.

As in the case $g=0$, the canonical symplectic form $\omega=dp_1\wedge dx_1+dp_2\wedge dx_2$ induces
the volume forms $\Omega_+$ and $\Omega_-$ on $M_h^+$ and $M_h^-$, respectively, that are preserved by $\Theta_\pm^h$.

We derived the formulae for the map  $\Psi$ for $h<-\ell^2\rho$ in Theorem \ref{th:prva} and  the map $\Phi$ for $h<0$ in Theorem \ref{th:druga}. As a result,
we obtain  explicit algebraic forms of the return maps $\Theta_+^h$ and $\Theta_-^h$ for $h<-\ell^2\rho$ (see Figure \ref{fig:haos}).
\begin{figure}[h]
{\centering{\includegraphics[width=8cm]{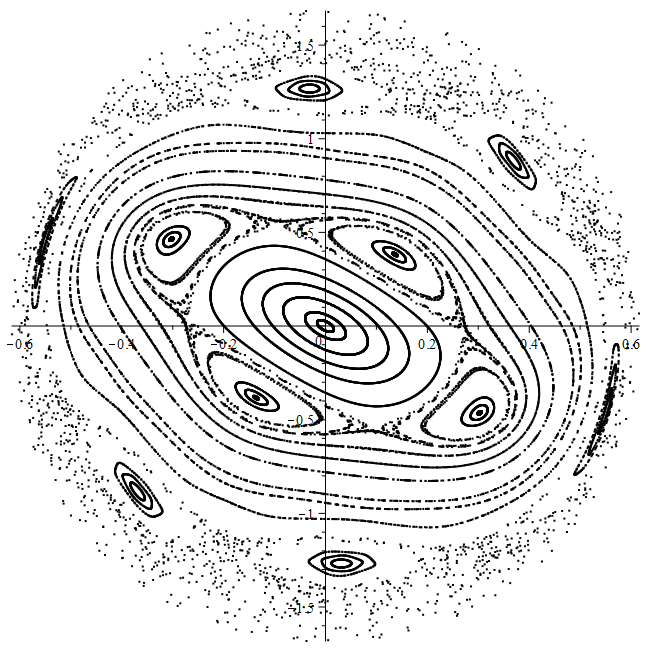}}
\caption{\label{fig:haos}  The trajectories of the return map $\Theta_+^h$ in the coordinates $x_1$ (abscissa) and $S$ (ordinate,  see Remark \ref{S=s} to recall the definition of $S$) for $\ell=1$, $g=7$, $\rho=5$, $H=h=-5.5$ after 1000 iterations.}}
\end{figure}
To illustrate the dynamics,  we present the trajectories of the  system, in the case $g=7, \rho=5, \ell=1$,  for several initial conditions
with the fixed value of the Hamiltonian $H=h=-5.5$ satisfying the condition \eqref{uslov:energija} (see Figure \ref{fig:razne}).
Inside the circle $\partial D$ are  blue arcs of parabolas, while outside the circle are  red arcs of ellipses.

\begin{rem}
 Note that we can apply the formulae for $\Psi$ and $\Phi$ given in Theorems \ref{th:prva} and \ref{th:druga} whenever there are up to two intersections of ellipses and parabolas with the boundary circle $\partial D$, regardless of the energy conditions. In Figure \ref{fig:lepa} (left), we have a trajectory with the same parameters
 $g=7, \rho=5, \ell=1$, but for the value of the Hamiltonian $H=h \approx 12$ that does not satisfy the condition \eqref{uslov:energija}. After 21 iterations of the map $\Theta_+^h$, the ellipse  intersects $\partial D$ in four points and Theorem \ref{th:prva} cannot be applied any more  to determine the next iteration.
\end{rem}
\begin{figure}[h]
{\centering{\includegraphics[width=13.5cm]{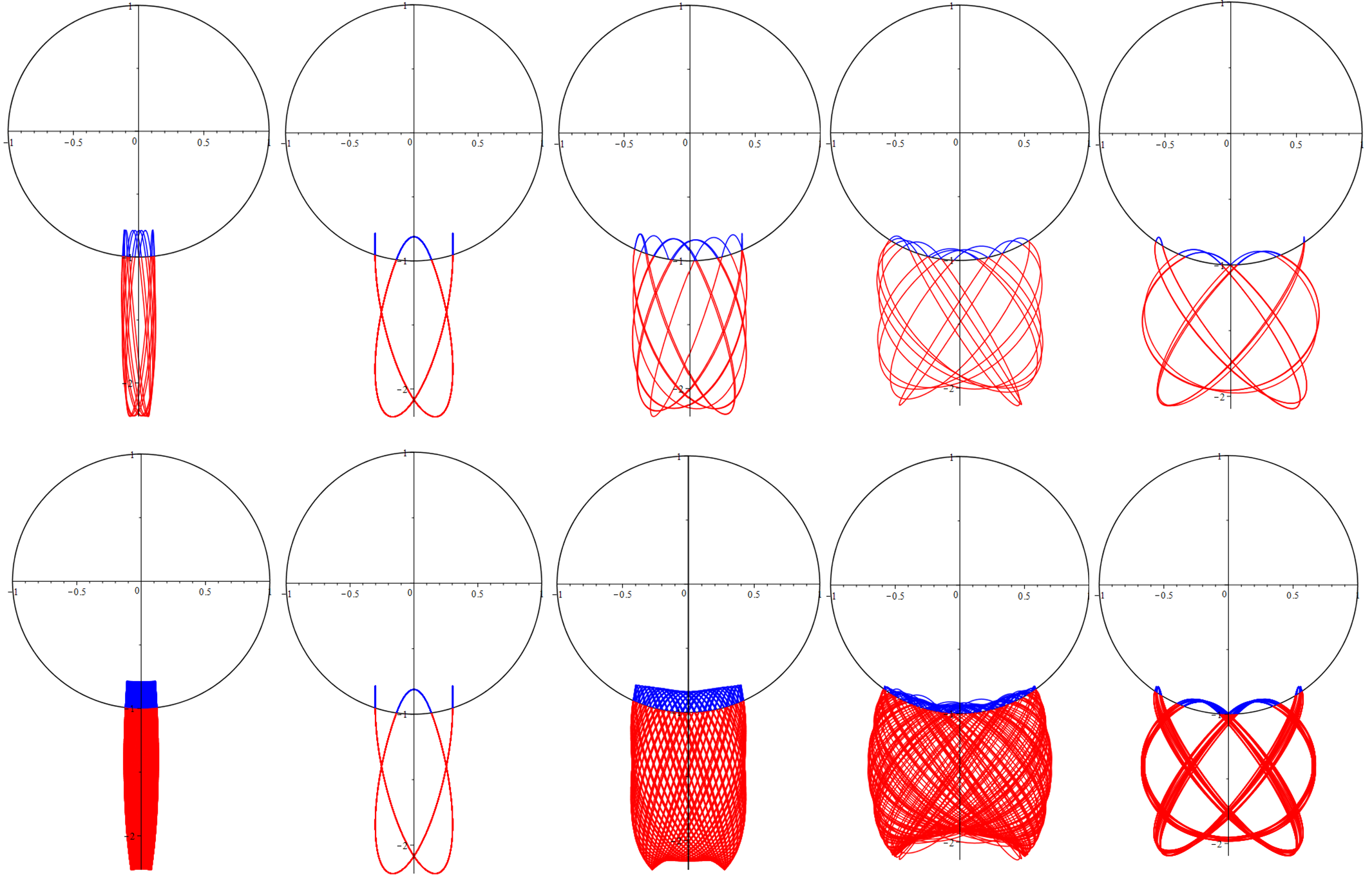}}
\caption{Five trajectories of the system  with the $C^0$-Hamiltonian \eqref{eq:ham2}, that correspond to trajectories of the return map $\Theta_+$ given in Figure \ref{fig:haos}, respectively from the left to the right:
an invariant  circle near the equilibrium at the origin; a $4$-periodic trajectory within ``four islands"; an invariant cycle away from the origin and before a claud trajectory;  a claud trajectory; a trajectory that corresponds to ``six islands". In the first  row, each case presents ten iterations, and in the second row one hundred iterations of  $\Theta_+$.
\label{fig:razne}}}
\end{figure}

\subsection{$C^0$-gluing of two non-commutatively integrable systems}

The $C^0$-approximative system can be seen as a $C^0$-gluing of two non-commutatively integrable, aka super-integrable, systems: the motion in a homogeneous gravitational field in $\R^2$ with noncompact invariant manifolds and the motion under the influence of the Hook potential centered at $C(0,-{g}/{\rho})$ with compact invariant manifolds over the boundary $\Delta$.
The first system has a complete sets of integrals
\[
F_1=p_1p_2+gx_1,\quad F_2=\frac12 p_2^2+gx_2,\quad F_3=p_1,\quad  F_4=2p_1(x_1p_2-x_2p_1)+gx_1^2
\]
(the Hamiltonian is $H_1=F_2+\frac12F_3^2$) and the second one has the complete set of integrals
\begin{equation}\label{eq:Gi}
\begin{gathered}
G_1=p_1p_2+gx_1+\rho x_1x_2, \qquad G_2=\frac12p_1^2+\frac12\rho x_1^2,\\
G_3=\frac12p_2^2+\frac12\rho x_2^2+gx_2, \qquad  G_4=\rho(x_1p_2-x_2p_1)-{g}p_1
\end{gathered}
\end{equation}
(the Hamiltonian is $H_2=G_2+G_3-\frac{\rho}2\ell^2$),
with the Poisson brackets:
\begin{align*}
 & \{F_1,F_2\}=-gF_3, \quad \{F_1,F_3\}=g, \quad \{F_1,F_4\}=-4F_2F_3+2F_3^2,\\
 &\{F_2,F_3\}=0,\quad \{F_2,F_4\}=2F_1F_3, \quad \{F_3,F_4\}=-2F_1,\\
 & \{G_1,G_2\}=-G_4, \quad \{G_1,G_3\}=G_4, \quad
  \{G_1,G_4\}=2\rho G_3-2\rho G_2-{g^2},\\
 &\{G_2,G_3\}=0,\quad \{G_2,G_4\}=-\rho G_1, \quad \{G_3,G_4\}=\rho G_1.
\end{align*}

The $C^0$-approximative system is defined by the Hamiltonian $H=H_1=F_2+\frac12F_3^2$ for $x_1^2+x_2^2\le \ell^2$ and  by the Hamiltonian
$H=H_2=G_2+G_3-\frac{\rho}2\ell^2$ for $x_1^2+x_2^2\ge \ell^2$, such that $H_1\vert_\Delta=H_2\vert_\Delta$.
If we can construct another $C^\infty$-smooth functions $K_1(F_1,F_2,F_3,F_4)$ and
$K_2(G_1,G_2,G_3,G_4)$,  such that
\[
K_1(F_1,F_2,F_3,F_4)\vert_\Delta=K_2(G_1,G_2,G_3,G_4)\vert_\Delta,
\]
then the function
\begin{equation*}  K(x,p)=
\begin{cases}
        K_1(x,p), & \text{for}\quad   x_1^2+x_2^2\le\ell^2,\\
        K_2(x,p), & \text{for}\quad  x_1^2+x_2^2\ge\ell^2
    \end{cases}
\end{equation*}
would be a $C^0$-integral of our system. As a result, the set $\Delta$ would be foliated on invariant curves of our system,
given by the equations $\{H=h, K=k\}$.  However, although there are a lot of invariant circles, Figure \ref{fig:haos} indicates that this is not the case.

\subsection{Lyapunov stability of an elliptic fixed point of the return maps}

As the coordinates on $M_h^\pm$, we take the area integral $S=p_\varphi$ and the translation of the angular variable $Y=\varphi-\frac{\pi}2$ (see \eqref{polarne} and use the inequality $x_2<0$ for $H=h<0$):
\begin{equation}\label{Lambda}
\Lambda_\pm\colon M_h^\pm \longrightarrow \R^2, \qquad (S,Y)=\Big(x_1p_2-x_2p_1,\arctan\Big(-\frac{x_1}{x_2}\Big)\Big).
\end{equation}

Vice verse, on $M_h^\pm$ we have
\begin{align}
\label{xovi}& x_1=\ell\sin Y, \qquad x_2=-\ell\cos Y.
\end{align}

Further, from the first integrals $S$, $H$, and inequalities $\langle x,p\rangle \gtrless 0$ on $M_h^\pm$, we obtain
\begin{align}
\label{veza1} & p_1\ell\cos Y+p_2\ell\sin Y=S \quad \Rightarrow \quad p_1=\frac{S}{\ell\cos Y}-p_2\tan Y,\\
\label{veza2} & p_2^2(1+\tan^2 Y)-2\frac{S\sin Y}{\ell\cos^2 Y}p_2+\Big(\frac{S^2}{\ell^2\cos^2 Y}-2h-2g\ell\cos Y\Big)=0,\\
\label{veza3} & \ell\sin Y\Big(\frac{S}{\ell\cos Y}-p_2\tan Y\Big)-p_2\ell\cos Y=S\tan Y- \frac{\ell}{\cos Y}p_2   \gtrless 0.
\end{align}

The solutions of the quadratic equation \eqref{veza2} are
\begin{align*}
&p_2^+=\frac{S\sin Y}{\ell}- \cos Y\sqrt{(2h+2g\ell\cos Y)-\frac{S^2}{\ell^2}},\\
&p_2^-=\frac{S\sin Y}{\ell}+ \cos Y\sqrt{(2h+2g\ell\cos Y)-\frac{S^2}{\ell^2}}.
\end{align*}

Due to the  inequalities \eqref{veza3}, we get the mappings $\lambda_\pm$ inverse to \eqref{Lambda} on the images $\Lambda_\pm(M_H^\pm)$, that is, the parametrization of $M_h^+$ and $M_h^-$ in variables $(S,Y)$ given by \eqref{xovi} and
 \begin{align*}
&p_2=p_2^+, \qquad p_1=\frac{S}{\ell\cos Y}-p_2^+\tan Y,\\
&p_2=p_2^-, \qquad p_1=\frac{S}{\ell\cos Y}-p_2^-\tan Y,
\end{align*}
respectively. The symplectic form $\omega$ (see \eqref{kanonska}) restricted to $M_h^+$ and $M_h^-$ takes the standard form:
\[
\lambda^*_\pm\omega=dS \wedge dY.
\]

The return maps $\Theta_+^h$ and $\Theta_-^h$, in the coordinates $(S,Y)$ are given by
\begin{equation}\label{eq:Theta+-}
\Theta_+^h=\Lambda_+\circ \Phi\circ \Psi \circ \lambda_+, \qquad \Theta_-^h=\Lambda_-\circ \Psi\circ \Phi \circ \lambda_-.
\end{equation}

We will consider the return map $\Theta_+^h$, while $\Theta_-^h$ can be treated in the same way. The point $(0,0)$ is a fixed point of the map:
\begin{equation}\label{mapseries}
\begin{aligned}
(0,0)&\longmapsto (0,-\ell,0,-\sqrt{2h+2g\ell}) \longmapsto (0,-\ell,0,\sqrt{2h+2g\ell}) \\
&\longmapsto (0,-\ell,0,-\sqrt{2h+2g\ell}) \longmapsto (0,0).
\end{aligned}
\end{equation}

\begin{lem}
The differentials $d\lambda_+$ and  $d\Lambda_+$ are given by the formulae
\begin{align}
&d\lambda_+ =
\small
\left(
\begin{array}{cc}\label{dvadifer}
0 & \ell \cos(Y) \\
0 & \ell \sin(Y) \\
\frac{\cos Y\ell G-S\sin Y}{\ell^2
G} & -\frac{S \ell\sin Y
G  + g\ell^3 - 3 \cos^2 Y g\ell^3+
(S^2- 2h\ell^2) \cos Y +g\ell^3}{\ell^2 G} \\
\frac{\sin Y \ell G + S \cos Y }{\ell^2 G}
& \frac{S \cos Y G \ell +
\sin Y ( 2h\ell^2 + 3 \cos Y  g\ell^3  - S^2)}{\ell^2 G}
\end{array}
\right),\\
&d\Lambda_+ =
\small
\left(
\begin{array}{cccc}\label{dvadifer1}
p_2 & -p_1 & -x_2 & x_1 \\
-\frac{x_2}{\ell^2}   &  \frac{x_1}{\ell^2}   &   0  &  0
\end{array}
\right),
\end{align}
where $G=\sqrt{2h+2g\ell\cos Y-\frac{S^2}{\ell^2}}$.
\end{lem}
A proof follows by direct calculations.

\begin{exm}\label{ex:stability} The  formulae for the differentials $d\Psi$ and $d\Phi$ are quite complicated.
In order to simplify the calculations,  as in previous examples in  Section \ref{sec3}, let us fix the parameters of the system \eqref{eq:ham2}: $\ell=1, g=7, \rho=5$ and the energy level $h=-5.5$. All calculations are rounded up to three decimal places.
The formulae for the differentials $d\Psi$ and $d\Phi$ become:

$$\begin{aligned}
d\Psi\vert_{(0,-\ell,0,-\sqrt{2h+2g\ell})}&=\left( \begin{matrix}-0.579&0&-0.365&0\\
0&1&0&0\\
1.823&0&-0.579&0\\
0&0&0&-1
\end{matrix}\right),\\
d\Phi\vert_{(0,-\ell,0,\sqrt{2h+2g\ell})}&=\left( \begin{matrix}1&0&0.495&0\\
0&1&0&0\\
0&0&1&0\\
0&0&0&-1
\end{matrix}\right).
\end{aligned}
$$

Using \eqref{dvadifer} and \eqref{dvadifer1}, we get the Jacobian matrix of the map $\Theta^h_+$ from \eqref{eq:Theta+-} at the point $(0,0)$:
\begin{equation}\label{jacobi}
 d\Theta^h_+\vert_{(0,0)}=\left(\begin{matrix}0.549&2.214\\
-0.651&-0.805
\end{matrix}\right),
\end{equation}
with the eigenvalues
$\lambda_{1,2}=-0.128\pm 0.992\,\i$. Since $|\lambda_1|=1$, and $\lambda_1\ne 1$, it follows that the point $(0,0)$ is a nondegenerate elliptic fixed point of the return  map $\Theta^h_+$.

The behaviour of the trajectories presented in the Figure \ref{fig:haos} suggests that the fixed point $(0,0)$ is nonlinearly stable. In order to verify this claim, we will study the nonlinear terms, to prove that the elliptic fixed point is nonresonant of order at least four, and to calculate the Moser twist coefficient. More details about the underlying theory can be found in \cite{Kr, KP, BM, Mo, JZh}. If the Moser twist coefficient is not equal to zero, then the non-resonant fixed point is Lyapunov stable and most of the trajectories in its neighborhood belong to invariant circles \cite{Kr}.

Thus, we calculate the Moser twist coefficient. For the eigenvalues of the matrix \eqref{jacobi}, one gets
 $\lambda_1^2=-0.967-0.254\,\i$, $\lambda_1^3=0.375-0.927\,\i$,
$\lambda_1^4=0.871+0.490\,\i$. All  are different from $1$, hence  $\lambda_1$ is non-resonant up to order four.
The eigenvectors of the matrix \eqref{jacobi} are $\bold{v}_{1}$ and $\bold{v}_{2}$ given by:
\[
\bold{v}_{1}=\bar{\bold{v}}_2=(0.879, -0.269+0.394\,\i)^T
\]

 In the new coordinates
\[
\left(\begin{matrix}\tilde{S}\\ \tilde{Y}\end{matrix}\right)=\left(\begin{matrix}1.138&0\\
0.776&2.539
\end{matrix}\right)\left(\begin{matrix}S\\ Y\end{matrix}\right),
\]
the return map has a formula
\[
\begin{aligned}
\left(\begin{matrix}\tilde{S}_1\\ \tilde{Y}_1\end{matrix}\right)&=\left(\begin{matrix}-0.128&0.992\\
-0.992&-0.128
\end{matrix}\right)\left(\begin{matrix}\tilde{S}\\ \tilde{Y}\end{matrix}\right)+\\
&\left(\begin{matrix} -0.103\tilde{S}^3-0.478\tilde{S}^2\tilde{Y}+0.115\tilde{S}\tilde{Y}^2-0.186\tilde{Y}^3+o( v^3)\\
-0.158\tilde{S}^3+0.286\tilde{S}^2\tilde{Y}-0.846\tilde{S}\tilde{Y}^2+0.709\tilde{Y}^3+o(v^3)
\end{matrix}\right),
\end{aligned}
\]
where $ v=\sqrt{\tilde{S}^2+\tilde{Y}^2}$.
The linear term corresponds to the rotation about the origin, and the quadratic term is equal to zero. The formula for the Moser twist coefficient $\tau$, is given in terms of the coefficients of the return map written in the complex notation $Z=\tilde{S}+\tilde{Y}\,\i$. The formula for $\tau$ from \cite{Mo} is
$$
\tau=\frac{1}{\i}\Big(c_{21}+2|c_{20}|^2\Big[\frac{2\lambda_1+1}{\lambda_1-1}+\frac{1}{\lambda_1^3-1}\Big]\Big),
$$
where $c_{21}$ and $c_{20}$ are the coefficients  from the formula for the return map in the complex coordinates:
$$
Z_1=\lambda_1(Z+c_{20}Z^2+c_{11}Z\bar{Z}+c_{02}\Bar{Z}^2+c_{30}Z^3+c_{21}Z^2\bar{Z}+c_{12}Z\bar{Z}^2+c_{03}\bar{Z}^3)+o(r^3).
$$
Since the return map is given in the form
$$
\begin{aligned}
S_1&=\Re(\lambda_1) S-\Im(\lambda_1) Y+a_{20}S^2+a_{11}SY+a_{02}Y^2\\
&+a_{30}S^3+a_{21}S^2Y+a_{12}SY^2+a_{03}Y^3+o(r^3),\\
Y_1&=\Im(\lambda_1) S+\Re(\lambda_1) Y+b_{20}S^2+b_{11}SY+b_{02}Y^2\\
&+b_{30}S^3+b_{21}S^2Y+b_{12}SY^2+b_{03}Y^3+o(r^3),\\
\end{aligned}
$$
we calculate
$$
\begin{aligned}
c_{21}&=\frac{1}{8\lambda_1}\Big(3a_{30}+b_{21}+a_{12}+3b_{03}+(3b_{30}-a_{21}+b_{12}-3a_{03})\i\Big),\\
c_{20}&=\frac{1}{4\lambda_1}\Big(a_{20}-a_{02}+b_{11}+(b_{20}-b_{02}-a_{11})\i\Big).
\end{aligned}
$$
Using that the quadratic term is zero, the formula for $\tau$ reduces to
$$
\tau=\frac{1}{\i}c_{21}.
$$
We get $\tau=0.280$. Thus, the origin as a fixed point is indeed Lyapunov nonlinearly stable.
\end{exm}

\section{$C^0$-system without the gravitational field}\label{sec4}

As in Section \ref{sec2}, it is beneficiary to describe the case  of the $C^0$-approximative system without the gravitational field.
 As a result, we obtain a $C^0$-gluing of two super-integrable systems with first integrals \eqref{eq:Fi} and \eqref{eq:Gi} (by setting $g=0$ in \eqref{eq:Gi}), which has a remarkable possibility of applying the Liouville-Arnol'd theorem.

Set $g=0$ in the Hamiltonian for the $C^0$-approximation:
\begin{equation}\label{eq:h3}
H(x,p)=\frac12 \langle p,p\rangle+\frac12\rho(\vert x\vert^2-{\ell^2})\theta(\vert x\vert-\ell).
\end{equation}

Now, we can define the dynamics for all the initial conditions.

Namely, consider an initial condition $(x(0),p(0))=(x_0,p_0)\in \Delta_0$: $\langle p_0,x_0\rangle=0$, $\langle x_0,x_0\rangle={\ell^2}$.
If $\langle p_0,p_0\rangle={\ell^2}\rho$, then the solution $x(t)$ of \eqref{ham2g} with $g=0$ is the uniform circular motion along $\partial D$, see \eqref{kruznice} for $\lambda^2=\rho$, according to the rule \textbf{R1} for $g=0$. For $\langle p_0,p_0\rangle>{\ell^2}\rho$, the solution $x(t)$ is the motion along the ellipse centered at $O(0,0)$ that is tangent to $\partial D$ at $x_0$ and $-x_0$, according to the rule \textbf{R1} for $g=0$. For $p_0=0$, the trajectory is the equilibrium position, according to the rule \textbf{R2} for $g=0$. However, for
$0<\langle p_0,p_0\rangle<{\ell^2}\rho$,
the initial position $(x_0,p_0)$ belongs to $\Pi$: neither equations \eqref{ham1g} nor equations \eqref{ham2g} can be utilized to describe the motion.  Instead, by definition, we extend the dynamics to the entire interval
\[
0 \le \langle p_0,p_0\rangle \le {\ell^2}\rho,
\]
and define that the trajectory
is again the uniform motion along the circle $\partial D$, in the complex notation:
\begin{equation}\label{kruznice}
\mathbf x(t)=\mathbf x_0 e^{\i\lambda t}, \quad \mathbf p(t)=\mathbf p_0e^{\i\lambda t}, \quad \mathbf p_0=\i\lambda\mathbf x_0, \quad 0<\lambda^2\le \rho.
\end{equation}

The Hamiltonian function $H$ and the area integral $S$
remain to be preserved along the motion and we can use the polar coordinates to integrate the Hamiltonian system defined by \eqref{eq:h3} by quadratures. The complete system can be integrated in the same way as a $C^1$-smooth system. All ``$C^1$-statements" should be replaced by ``$C^0$-statements", while ``$C^{\infty}$-smooth statements" remain ``$C^{\infty}$-smooth statements".

Here we need to take into account that the
minimum of the effective potential,
\begin{equation}\label{C0efektivni}
 V_s(r)=\begin{cases}
        \frac{1}2\frac{s^2}{r^2}, & \text{for}\quad 0<r\le\ell,\\
        \frac{1}2\frac{s^2}{r^2}+\frac{1}2\rho r^2-\frac{1}2\rho {\ell^2}, & \text{for}\quad  r\ge\ell,
    \end{cases}
\end{equation}
for $\vert s\vert \le \sqrt{\rho}{\ell^2}$, is at the point of non-smoothness $r_s^*=\ell$  (see Figure \ref{fig:potencijali}). The corresponding trajectories in the non-reduced phase space are given by \eqref{kruznice}. For $\vert s\vert > \sqrt{\rho}{\ell^2}$, the minimum $r^*_s$ is the solution of the equation
$\rho (r^*_s)^4=s^2$, that is, $r^*_s=\rho^{-1/4}\sqrt{\vert s\vert}>\ell$. It  corresponds to uniform circular motion of radius $r^*_s$
($\langle p(t),p(t)\rangle=\rho (r^*_s)^2$) satisfying
\[
h=\frac12\langle p(t),p(t)\rangle+\frac12\rho (r^*_s)^2-\frac12 \rho {\ell^2}=\sqrt{\rho}\vert s\vert -\frac12{\rho} {\ell^2}.
\]
\begin{figure}[h]\label{fig3}
{\centering{\includegraphics[width=10cm]{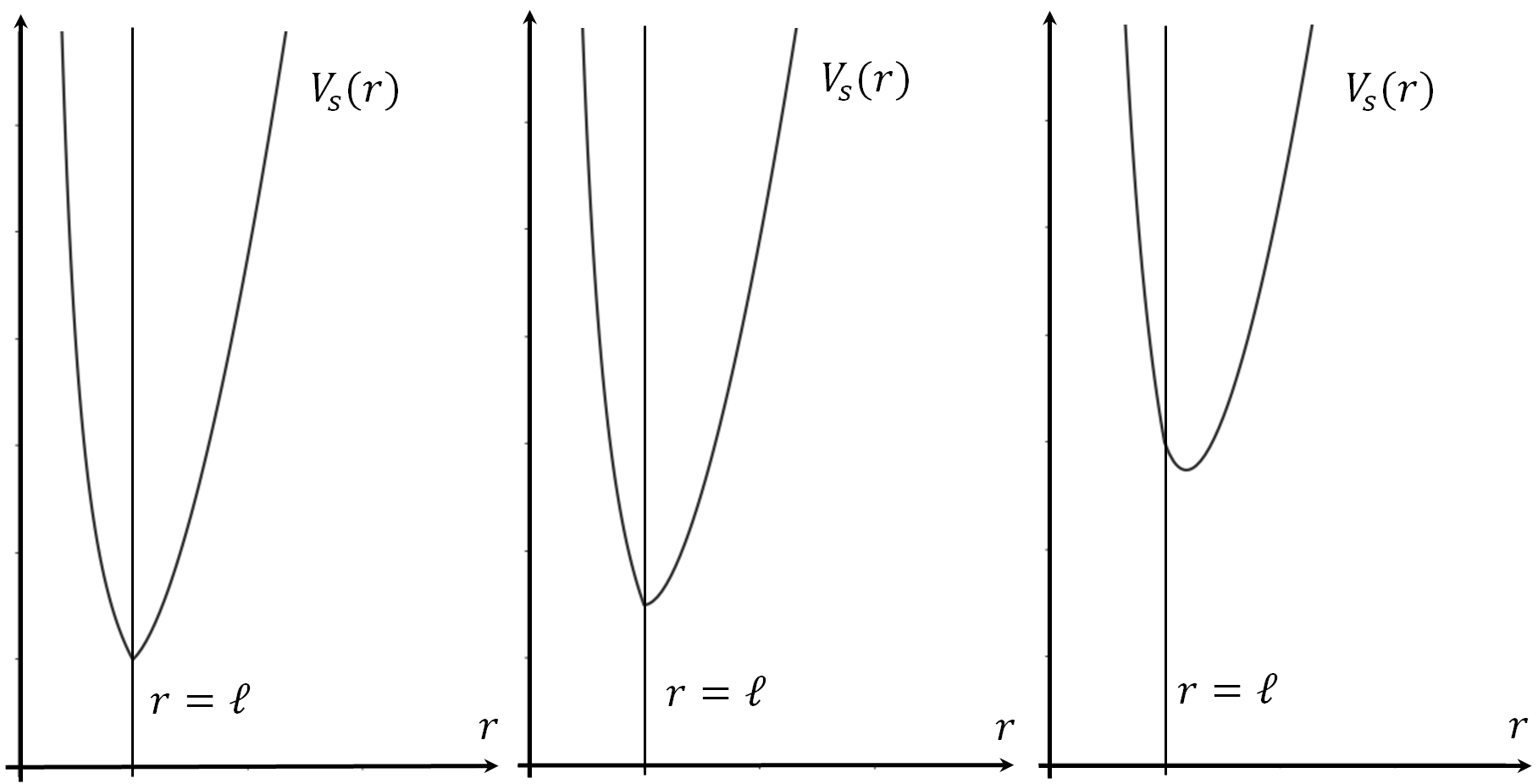}}
\caption{The effective potential for the $C^0$-approximative system, the cases with $\vert s\vert < \sqrt{\rho}{\ell^2}$, $\vert s\vert = \sqrt{\rho}{\ell^2}$,
and $\vert s\vert > \sqrt{\rho}{\ell^2}$, respectively.\label{fig:potencijali}}}
\end{figure}
Therefore,
the set of regular values ${\mathcal Reg}$ is defined by
\[
{\mathcal Reg}=\{(h,s)\,\vert\,  s\ne 0,\,
h> \frac{s^2}{2{\ell^2}} \, \big(\vert s\vert \le \sqrt{\rho}{\ell^2}\big),\,
h> \sqrt{\rho}\vert s\vert -\frac{\rho}2 {\ell^2}\,\big(\vert s\vert \ge \sqrt{\rho}{\ell^2}\big) \}.
\]

We proceed with the description of the return maps $\Theta_\pm$ without restrictions on the value of the energy.
 The return maps are the rotations by the angle $\Delta\varphi_{h,s}$ (see \eqref{prirastaj} for $S=s\ne 0$
and the $C^0$-effective potential \eqref{C0efektivni}, and $\Delta\varphi_{h,0}=0$ for $S=0$). However, now $\Delta\varphi_{h,s}$ will be given in the algebraic form as well.
We use the complex notation, where
\[
S(x,p)=x_1p_2-x_2p_1=\frac{\i}{2}(\mathbf x\bar{\mathbf p}-\mathbf p\bar{\mathbf x}) \quad \text{and}\quad
J(x,p)=\langle x,p\rangle=\frac12(\mathbf x\bar{\mathbf p}+\mathbf p\bar{\mathbf x}).
\]

 As in the $C^1$-approximative case, let $\Sigma_-\subset T^*\R^2$ be the open region consisting of the trajectories with two regimes: $S=0$, or $S\ne 0$ and $V_s(\ell)<h$
(see \eqref{sigma-}). The map $\Phi$ is the same, see \eqref{PhiC1}. For the map $\Psi$, let us consider an initial condition $(x(0),p(0))\in\Delta_+$, i.e. $\vert x(0)\vert =\ell$ and $\langle p(0),x(0)\rangle >0$. As in the case $g>0$, we need to find the first value of time $\tau>0$, such that the solution
\begin{equation}\label{van0}
\mathbf x(t)=\mathbf y_1e^{\i\sqrt{\rho}t}+\mathbf y_2e^{-\i\sqrt{\rho}t}, \quad \mathbf y_1=\frac12\Big(\mathbf x_0-\frac{\i}{\sqrt{\rho}}\mathbf p_0\Big), \quad \mathbf y_2=\frac12\Big(\mathbf x_0+\frac{\i}{\sqrt{\rho}}\mathbf p_0\Big),
\end{equation}
satisfies the equation
\[
\mathbf x(\tau)\bar{\mathbf x}(\tau)={\ell^2}.
\]

Note that since $x(0)$ and $p(0)$ are not orthogonal, both $\mathbf y_1$ and $\mathbf y_2$ are different from zero.

\begin{lem}\label{lema2} The time $\tau$ satisfies the equation:
\[
e^{\i\tau\sqrt{\rho}}=\frac{\mathbf y_2\bar{\mathbf y}_1}{\vert\mathbf y_1\vert\vert{\mathbf y_2}\vert}.
\]
\end{lem}

\begin{proof} From \eqref{polinom4}, for $g=0$, we obtain that $z=e^{\i\tau\sqrt{\rho}}$ is a solution of the bi-quadratic equation
\begin{equation}\label{polinom2x2}
a z^4-(a+\bar a)z^2+\bar a=0, \qquad a=\mathbf y_1\bar{\mathbf y}_2.
\end{equation}

The roots of \eqref{polinom2x2} are
$\pm 1$ and $\pm \bar a/\vert a\vert$.
They are of modulus $1$ and correspond to the four intersections, with multiplicities in the case of the segment, of   $\mathcal E(\mathbf x_0,\mathbf p_0)$
and $\partial D$.  We need to order the solutions along the trajectory $\mathbf x(t)$, such that
$z_1=1$ (corresponding to $t_1=0$); $z_2=e^{\i t_2\sqrt{\rho}}$ ($0<t_2=\tau<\pi/\sqrt{\rho}$); $z_3=-1$ ($t_3=\pi/\sqrt{\rho}$); $z_4=e^{\i t_4\sqrt{\rho}}$ ($\pi/\sqrt{\rho}<t_4=\tau+\pi/\sqrt{\rho}<2\pi/\sqrt{\rho}$). Thus, we obtain two options for $z_2$:
\[
z_2\in \{\bar a/\vert a\vert,-\bar a/\vert a\vert\}.
\]
Since $\Im(z_2)>0$ and
\begin{align*}
\Im(4\bar a)=&\Im\Big((\mathbf x_0+\frac{\i}{\sqrt{\rho}}\mathbf p_0)(\bar{\mathbf x}_0+\frac{\i}{\sqrt{\rho}}\bar{\mathbf p}_0)\Big)\\
=& \Im\Big(\frac{\i}{\sqrt{\rho}}(\mathbf p_0\bar{\mathbf x}_0+{\mathbf x}_0\bar{\mathbf p}_0)\Big)=\frac{2}{\sqrt{\rho}}\langle x(0),p(0)\rangle>0,
\end{align*}
we get that $z_2=e^{\i\tau\sqrt{\rho}}=\bar a/\vert a\vert$.
\end{proof}

\begin{thm}\label{integrabilniBilijar} The return map $\Theta_+\colon \Delta_+\to \Delta_+$ in the complex notation
is given by
\begin{align*}
&\mathbf x_2=\frac12\Big(\mathbf x_0+\frac{\i}{\sqrt{\rho}}\mathbf p_0\Big)K(\mathbf x_0,\mathbf p_0)+\frac12\Big(\mathbf x_0-\frac{\i}{\sqrt{\rho}}\mathbf p_0\Big)K(\mathbf x_0,\mathbf p_0)^{-1}
+\frac{J(\mathbf x_0,\mathbf p_0)}{H(\mathbf p_0)}\mathbf p_1,\\
&\mathbf p_2=\frac{\i\sqrt{\rho}}2\Big(\mathbf x_0+\frac{\i}{\sqrt{\rho}}\mathbf p_0\Big)K(\mathbf x_0,\mathbf p_0)-\frac{\i\sqrt{\rho}}2\Big(\mathbf x_0-\frac{\i}{\sqrt{\rho}}\mathbf p_0\Big)K(\mathbf x_0,\mathbf p_0)^{-1},
\end{align*}
where
\[
K(\mathbf x_0,\mathbf p_0)=\sqrt{\frac{{\ell^2}+\frac{2}\rho H(\mathbf p_0)+\frac{2}{\sqrt{\rho}}S(\mathbf x_0,\mathbf p_0)}{{\ell^2}+\frac{2}\rho H(\mathbf p_0)-\frac{2}{\sqrt{\rho}}S(\mathbf x_0,\mathbf p_0)}}.
\]

The functions
$S$, $J$, and $H=\frac12\mathbf p\bar{\mathbf p}$ are first integrals of  $\Theta_+$.
The return map $\Theta_+^h$ restricted to $M_h^+=\{H=h\}\cap \Delta_+$ is integrable.
\end{thm}

Note that, as in the $C^1$-case with $g=0$, there is the relation $S^2+J^2=2{\ell^2}H$ and $J$ is not a first integral of the  continuous system.

\begin{proof}
Let $(\mathbf x_1,\mathbf p_1)=\Psi(\mathbf x_0,\mathbf p_0)$ (see Figure \ref{fig:prsten}). According to Lemma \ref{lema2}, we have
\[
\mathbf x_1=\mathbf x(\tau)=
\mathbf y_1\frac{\mathbf y_2\bar{\mathbf y}_1}{\vert\mathbf y_1\vert\vert{\mathbf y_2}\vert}+
\mathbf y_2\frac{\mathbf y_1\bar{\mathbf y}_2}{\vert\mathbf y_1\vert\vert{\mathbf y_2}\vert}=
\mathbf y_2\frac{\vert\mathbf y_1\vert}{\vert{\mathbf y_2}\vert}+\mathbf y_1\frac{\vert\mathbf y_2\vert}{\vert{\mathbf y_1}\vert}.
\]

Next, since
\[
\mathbf p(t)=\frac{d}{dt}\mathbf x(t)=i\sqrt{\rho}\mathbf y_1e^{\i\sqrt{\rho}t}-i\sqrt{\rho}\mathbf y_2e^{-\i\sqrt{\rho}t},
\]
we get
\[
\mathbf p_1=\mathbf p(\tau)=
\i\sqrt{\rho}\mathbf y_2\frac{\vert\mathbf y_1\vert}{\vert{\mathbf y_2}\vert}-\i\sqrt{\rho}\mathbf y_1\frac{\vert\mathbf y_2\vert}{\vert{\mathbf y_1}\vert}.
\]
Note that
\[
J(\mathbf x_1,\mathbf p_1)=\frac12(\bar{\mathbf x}_1\mathbf p_1+\mathbf x_1\bar{\mathbf p}_1)=-\frac12(\bar{\mathbf x}_0\mathbf p_0+\mathbf x_0\bar{\mathbf p}_0)=-J(\mathbf x_0,\mathbf p_0).
\]

The map $\Phi$ is the same as in the $C^1$ case,   see \eqref{PhiC1}. Let $(\mathbf x_2,\mathbf p_2)=\Phi(\mathbf x_1,\mathbf p_1)$:
\[
\mathbf x_2=\mathbf x_1+ \tau_1 \mathbf p_1,  \quad \mathbf p_2= \mathbf p_1, \quad   \tau_1=
-\frac{\bar{\mathbf x}_1\mathbf p_1+\mathbf x_1\bar{\mathbf p}_1}{\mathbf p_1\bar{\mathbf p}_1}=\frac{J(\mathbf x_0,\mathbf p_0)}{H(\mathbf p_0)}.
\]
Therefore,
\[
J(\mathbf x_2,\mathbf p_2)=J(\mathbf x_1,\mathbf p_1)+\tau_1 \mathbf p_1\bar{\mathbf p}_1=J(\mathbf x_0,\mathbf p_0).
\]

Finally,  from the identities
\begin{align*}
&\vert \mathbf y_1\vert =\frac12 \sqrt{(\mathbf x_0-\frac{i}{\sqrt{\rho}}\mathbf p_0)(\bar{\mathbf x}_0+\frac{i}{\sqrt{\rho}}\bar{\mathbf p}_0)}
=\frac12\sqrt{{\ell^2}+\frac{2}\rho H(\mathbf p_0)+\frac{2}{\sqrt{\rho}}S(\mathbf x_0,\mathbf p_0)},\\
&\vert \mathbf y_2\vert =\frac12 \sqrt{(\mathbf x_0+\frac{i}{\sqrt{\rho}}\mathbf p_0)(\bar{\mathbf x}_0-\frac{i}{\sqrt{\rho}}\bar{\mathbf p}_0)}
=\frac12\sqrt{{\ell^2}+\frac{2}\rho H(\mathbf p_0)-\frac{2}{\sqrt{\rho}}S(\mathbf x_0,\mathbf p_0)}
\end{align*}
we get the required algebraic expression for $\Theta_+$, as stated in the formulation of theorem.

The mappings $\Phi$ and $\Psi$ are induced from the  $C^{\infty}$-smooth  Hamiltonian flows. Therefore, their composition $\Theta_+^h=\Phi\circ\Psi\vert_{M_h^+}$
restricted to $M_h^+=\{H=h\}\cap \Delta_+$ preserves the canonical symplectic form.
The dynamics is a rotation of invariant   circles
$\delta^+_{h,s}=M_h \cap \{S=s\}$ depending on the value $s$.
\end{proof}
\begin{figure}[h]
{\centering{\includegraphics[width=8cm]{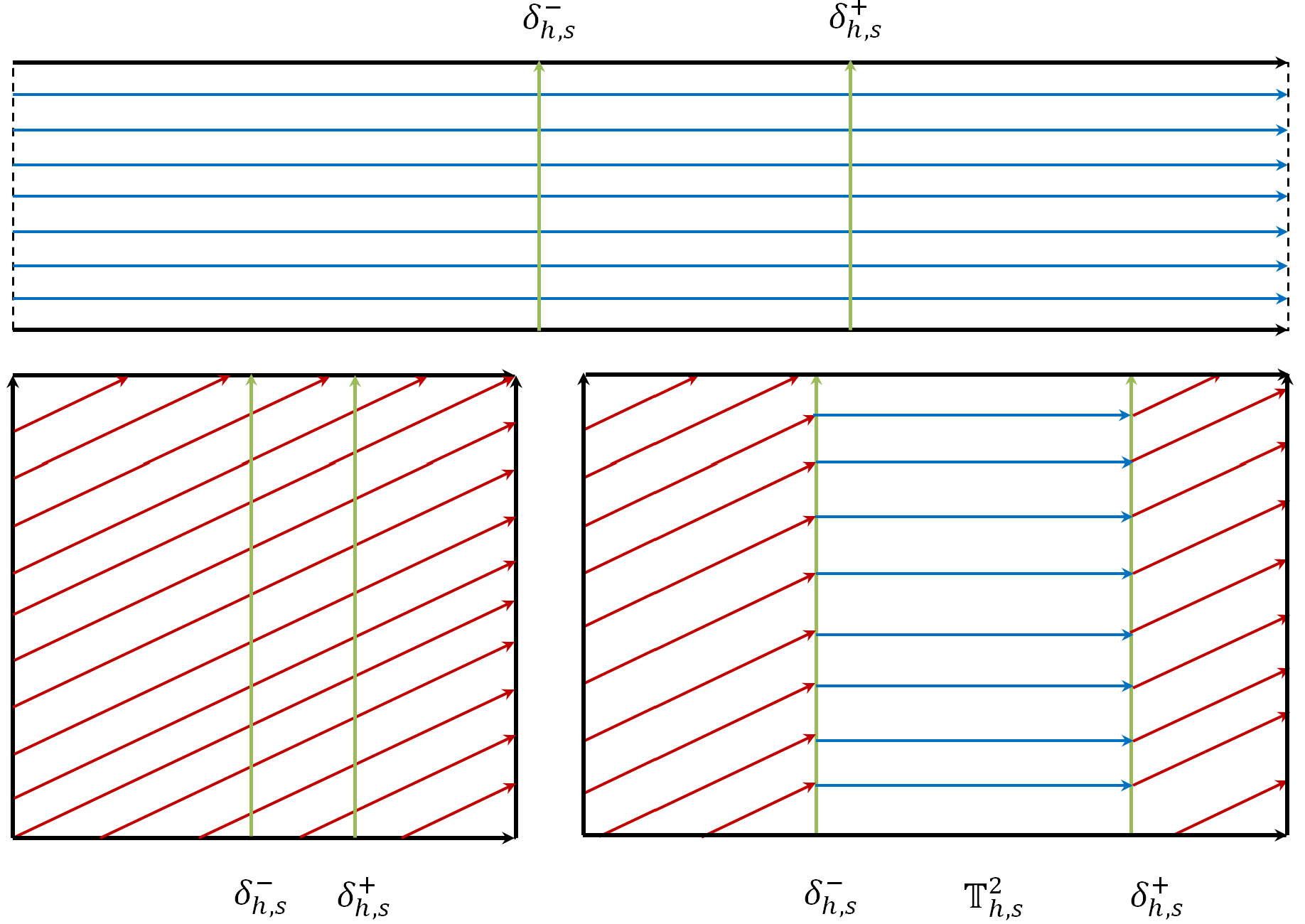}}
\caption{The invariant cylinder of the free motion (up), the resonant torus for the Hook potential (down, left), and $C^0$-invariant torus $\mathbb T^2_{h,s}$ with $C^0$-dynamics
(down, right). The return maps $\Theta_\pm$ are rotations of the invariant  circles $\delta^\pm_{h,s}$.\label{fig:lepljenje}}}
\end{figure}
At the level of $C^0$-gluing of two super-integrable systems, we have a similar situation as
illustrated in Figure \ref{fig:presek}. However, besides the fact that the invariant tori $\mathbb T^2_{h,s}$ are of the class $C^0$ along the intersections $\delta^\pm_{h,s}$, the difference is that, before gluing, all regular invariant tori defined by the Hamiltonian
with the Hook potential and the area integral $S$ are $2:1$-resonant and the system is super-integrable. After the gluing, the resulting tori $\mathbb T^2_{h,s}$ are generically non-resonant, which is schematically shown in Figure \ref{fig:lepljenje}.
Note also that the hypersurface $\Delta$ is an invariant variety of the Hamiltonian flow of the function $S$ and that circles $\delta^\pm_{h,s}$ are  trajectories of the corresponding Hamiltonian flow.

In the action-angle coordinates $(I_1,I_2,\psi_1,\psi_2)$ of the system, obtained via the same construction as in Theorem \ref{th:integrability}, the trajectories on
$\mathbb T^2_{h,s}$ smooth out and become straight lines.

\section{$C^1$ and $C^0$ integrability via gluing of integrable systems}\label{sec5}

\subsection{Definitions}

In this concluding Section,
we formalize and generalize the constructions presented in Sections \ref{sec2}, \ref{sec3}, and  \ref{sec4}.
Let $M$ be a four-dimensional symplectic manifold and $\Delta$ a hypersurface, such that $M\setminus\Delta$ has two  connected components, say $M_1$ and $M_2$.
Consider two $C^\infty$-smooth Hamiltonian systems with Hamiltonian
functions $H_1$ and $H_2$ respectively, such that $H_1\vert_{\Delta}=H_2\vert_{\Delta}$, and that $H_1$ and $H_2$ are functionally independent and different almost everywhere in a neighborhood of $\Delta$. Let
\begin{equation*}  H(x)=
\begin{cases}
        H_1(x), & \text{for}\quad x\in M_1\cup\Delta,\\
        H_2(x), & \text{for}\quad  x\in M_2\cup\Delta.
    \end{cases}
\end{equation*}
If $dH_1(x)=dH_2(x)$  for all $x\in \Delta$, then $H$ is $C^1$-smooth, otherwise it is of a class $C^0$ on $\Delta$.

We additionally assume that $\Delta$ has a decomposition $\Delta_+\cup \Delta_-\cup \Delta_0$, where $\Delta_0$ is a  submanifold of  a maximal dimension two, such that $\Delta_+$ and $\Delta_-$ are transversal to the considered Hamiltonian vector fields  $X_{H_1}$ and $X_{H_2}$, and
at the points of $\Delta_+$, the Hamiltonian vector fields $X_{H_1}$ and $X_{H_2}$ are directed into $M_2$,
while at the points of $\Delta_-$, $X_{H_1}$ and $X_{H_2}$ are directed into $M_1$.

Then we can define a Hamiltonian system with $C^0$-Hamiltonian $H$ using the rules analogous to \textbf{R0}, \textbf{R1}, and \textbf{R2} given in Section \ref{sec3}.

\begin{itemize}
\item{} On $\Delta_+\cup\Delta_-$, there is a natural gluing of the flows of vector fields $X_{H_1}$ and $X_{H_2}$.

\item{} If $X_{H_1}(x_0)$ and  $X_{H_1}(x_0)$ are tangent to $\Delta$ and
the trajectories of $X_{H_1}$ and $X_{H_2}$
with the initial condition $x(0)=x_0\in\Delta_0$ belong to $M_1\cup\{x_0\}$ for $t\in(-\varepsilon,\varepsilon)$,  then the local flow is defined by $X_{H_1}$.

\item{} If $X_{H_1}(x_0)$ and  $X_{H_1}(x_0)$ are tangent to $\Delta$ and
the trajectories of $X_{H_1}$ and $X_{H_2}$
with the initial condition $x(0)=x_0\in\Delta_0$ belong to $M_2\cup\{x_0\}$ for $t\in(-\varepsilon,\varepsilon)$,  then the local flow is defined by $X_{H_2}$.

\item{} If $x_0\in\Delta_0$ is the equilibrium of both systems, then it is the equilibrium of the system with the Hamiltonian $H$.

\item{} At the remaining points of $\Delta_0$,  the local dynamics is not defined.

\end{itemize}

If both Hamiltonian systems defined by Hamiltonians $H_1$ and $H_2$ are integrable with additional $C^\infty$-smooth first integrals $K_1$ and $K_2$ respectively, such that $K_1\vert_{\Delta}=K_2\vert_{\Delta}$, then the function $K(x)$ defined by
\begin{equation*}K(x)=
\begin{cases}
        K_1(x), & \text{for}\quad  x\in M_1\cup\Delta,\\
        K_2(x), & \text{for}\quad  x\in M_2\cup\Delta,
    \end{cases}
\end{equation*}
is a $C^0$-first integral of the Hamiltonian system with Hamiltonian $H$.
In this case, we say that the Hamiltonian system with the Hamiltonian $H$ is \emph{$C^0$-integrable obtained by $C^0$-gluing of
integrable systems $(H_1,K_1)$ and $(H_2,K_2)$ along $\Delta$}.
If both $H$ and $K$ are of the class $C^1$, then we say that
the Hamiltonian system with the Hamiltonian $H$ is \emph{$C^1$-integrable obtained by $C^1$-gluing of
integrable systems $(H_1,K_1)$ and $(H_2,K_2)$ along $\Delta$}.

As in the case of the first integral $S$  in Sections \ref{sec2} and \ref{sec4}, we have:

\begin{lem}\label{posledica} Let the Hamiltonian system with the Hamiltonian $H$ be a \emph{$C^0$-integrable obtained by $C^0$-gluing of
integrable systems $(H_1,K_1)$ and $(H_2,K_2)$ along $\Delta$.}
Assume that  $K=K_1=K_2$
and that all  connected components of generic regular invariant manifolds $\{H_2=h,K=k\}$ are two-dimensional tori.
Then $\Delta$ is an invariant manifold of $X_{K}$ and the Hamiltonian flow of ${K}$ is super-integrable near $\Delta$  with closed  generic trajectories.
\end{lem}

\begin{proof}
Assume that there exists a set $W\subset \Delta$ of positive measure, where $X_{K}$ is transversal to $\Delta$.
Let $x_0\in W$ and consider the trajectory $x(t)$ of $X_K$ with the initial condition $x(t)=x_0$. Both Hamiltonian functions $H_1$ and $H_2$ are preserved along the Hamiltonian flow of $K$ and $H_1(x(t))=H_2(x(t))$. As a result, we obtain a set of positive measure where $H_1=H_2$, which contradicts our assumption.
Thus, the set $W$ where $X_K$ is transversal to $\Delta$ is of measure zero. Since $X_K$  is a $C^\infty$-smooth vector field, we conclude that $W=\emptyset$.

Since $X_{H_2}$ is not tangent to $\Delta_+\cup\Delta_-$, it is clear that regular invariant tori intersect   $\Delta_+\cup\Delta_-$ transversally.
Now, assume that $x_0\in \Delta$ belongs to a regular invariant torus $\mathbb T^2_{h,k}$, a component of $\{H_2=h,K=k\}$
that intersects transversally $\Delta$. Whence, the trajectory $x(t)$ is a  topological circle, a connected component of $\delta_{h,k}=\mathbb T^2_{h,k} \cap \Delta$.
On the other hand, in the action-angle coordinates near the torus $\mathbb T^2_{h,k}$, all the trajectories of $X_K$ are linearized. Since $x(t)$ is closed, all the other trajectories on $\mathbb T^2_{h,k}$ are closed as well, with the same resonance $m:n$. We apply the above arguments for a toric foliation in a toroidal neighborhood $U$ of $\mathbb T^2_{h,k}$ and conclude that  all the trajectories are  resonant with the same resonance $m:n$. Thus, besides $H_2$, there is an additional independent first integral  of the Hamiltonian system with the Hamiltonian $K$ within $U$. \end{proof}

Alternatively, if the functions $H_1$, $H_2$ and $K$ are independent, we directly have super-integrability of $X_K$.

Lemma \ref{posledica} implies  that a similar situation as presented in Figures \ref{fig:presek} and \ref{fig:lepljenje}  occurs here: the gluing of invariant tori of the Hamiltonian system defined by ${H_2}$ and invariant cylinders or tori of the Hamiltonian system defined by $H_1$ along the cycles of the system defined by $K$.

The above concept of $C^0$-integrability via $C^0$-gluing of integrable systems can be easily extend to Hamiltonian systems on symplectic manifolds of an arbitrary dimension.

\subsection{Natural mechanical systems}

The basic examples are natural mechanical systems on Riemannian manifolds $(Q,g)$ with two different potential functions $V_1$ and $V_2$. We assume that $V_1\vert_\Gamma=V_2\vert_\Gamma$, where $\Gamma$ is a hypersurface in $Q$, such that $Q\setminus\Gamma$ has two connected components $Q_1$ and $Q_2$.
Then $M=T^*Q$, $M_1=T^*Q_1$, $M_2=T^*Q_2$, $\Delta=\{(x,p)\in T^*Q\,\vert\, x\in \Gamma\}$, and
\begin{align*}
&\Delta_0=\{(x,p)\in\Delta\,\vert\, \xi=g(p)\in T_x\Gamma\},\\
&\Delta_+=\{(x,p)\in\Delta\,\vert\, \xi=g(p)\in T_x Q \, \text{is directed into} \, Q_2\},\\
&\Delta_-=\{(x,p)\in\Delta\,\vert\, \xi=g(p)\in T_x Q \, \text{is directed into} \, Q_1\}.
\end{align*}

As in Sections \ref{sec3} and \ref{sec4}, the projection of trajectories  to the configuration space $Q$  will be $C^1$-smooth at $\Gamma$.

\begin{exm}[The Lagrange top and the symmetric Euler top] Consider  the configuration space $Q=SO(3)$, a group of rotations in the Euclidean space $\R^3$, with a left-invariant Riemannian metric induced by the inertia operator $I=\diag(I_1,I_1,I_3)$ defining  the kinetic energy of a symmetric rigid body motion about a fixed point  (see \cite{Ar}). If the rigid body is placed in a homogeneous gravitational field, we get the Lagrange top with the potential function $V_1=mg \langle \vec{\chi},\vec{\gamma}\rangle=mg\chi\gamma_3$, where $m$ is the mass, $\vec{\chi}=(0,0,\chi)$ is the position of the center of the mass of the body and
$\vec{\gamma}=(\gamma_1,\gamma_2,\gamma_3)$ is the direction of the gravitational field in the body reference frame. Let $\Gamma\subset SO(3)$ be the set of position of the rigid body where $\langle \vec{\gamma},\vec{\chi}\rangle=0$.  The set $\Gamma$ is diffeomorphic to a 2-dimensional torus $\mathbb T^2$. Together with the Lagrange top, we consider the corresponding symmetric Euler top: $V_2\equiv  0$. Both systems are completely integrable with common commuting integrals $\langle \overrightarrow{\mathbf m},\vec{\gamma}\rangle$ and $m_3$, where $\overrightarrow{\mathbf m}=(m_1,m_2,m_3)$ is the angular momentum in the moving reference frame.
Let $Q_1$ be the set where $\langle \vec{\gamma},\vec{\chi}\rangle=0<0$ and $Q_2$ be the set where $\langle \vec{\gamma},\vec{\chi}\rangle=0>0$. In such a way,
by gluing of the Lagrange top and the Euler top, we get a $C^0$-integrable system on $SO(3)$.
\end{exm}

\begin{exm}[The Kepler problem and the Hook potential in $\R^n$]\label{kepler}
Let $B=\{x\in\R^n\,\vert\, \langle x,x\rangle<\ell^2\}$, $\Gamma=\partial B$, $B=Q_1$, $Q_2=\R^n\setminus (D\cup\partial B)$.
Consider the $C^1$-gluing of superintegrable systems
\begin{align*}
& H_1=\frac12\langle p,p\rangle-\frac{\kappa}{\vert x\vert},\\
& H_2=\frac12\langle p,p\rangle+\frac{1}{2}\sigma\langle x,x\rangle-\frac32\sigma\ell^2, \qquad \sigma=\frac{\kappa}{\ell^3}
\end{align*}
along $\Delta$ ($H_1\vert_\Delta=H_2\vert_\Delta$ and $dH_1\vert_\Delta=dH_2\vert_\Delta$). For $n=2$, this is a special case of the refraction billiard within circle
$\Gamma$ considered by
De Blasi and Terracini \cite{dBT}. The resulting system is $C^1$-noncommutatively integrable with 2-dimensional invariant manifolds. Besides the Hamiltonian function, we have the Noether integrals $\phi_{ij}=x_ip_j-x_jp_i$, $1\le i<j\le n$, that are common integrals both of the Kepler problem, and of the motion under the influence of the Hook potential. We can also consider the gluing with the Hamiltonian $H=H_1$ within $M_2$ and $H=H_2$ within $M_1$.
In such a way, we obtain a system without singularity at the origin.
\end{exm}

\begin{exm}[$C^0$-model of bungee in $\R^n$]
As a direct generalization of Section \ref{sec3},
consider the $C^0$-gluing of superintegrable systems
\begin{align*}
& H_1=\frac12\langle p,p\rangle+g x_n,\\
& H_2=\frac12\langle p,p\rangle+g x_n+\frac{1}{2}\rho\langle x,x\rangle-\frac12\rho\ell^2, \qquad \frac{g}{\rho}>\ell
\end{align*}
along $\Delta$ ($H_1\vert_\Delta=H_2\vert_\Delta$) with $\Gamma=\partial B$, $B=Q_1$, $Q_2=\R^n\setminus (D\cup\partial B)$ (see Example \ref{kepler}).
 The resulting system is $SO(n-1)$-invariant and has the Noether integrals $\phi_{ij}=x_ip_j-x_jp_i$, $1\le i<j\le n-1$.  Due to $SO(n-1)$-symmetry it follows that the for every solution $(x(t),p(t))$, there exist a matrix
 \[
 \mathbf R=
 \left(
\begin{array}{cc}
\mathbf R_0 & 0  \\
0 & 1
\end{array} \right), \qquad \mathbf R_0\in SO(n-1),
\]
such that the trajectory
 $(\mathbf R x(t),\mathbf R p(t))$ belongs to the space 6-dimensional symplectic linear subspace $\{(x_1,x_2,0,\dots,0,x_n,p_1,p_2,0,\dots,0,p_n)\}$.
Therefore, without losing a generality we may assume that $n=3$. When the value of the area integral $S=\phi_{12}=x_1p_2-p_1x_2$ is equal to zero, the problem reduces to a two-dimensional system studied in Section \ref{sec3}. For $S=s\ne 0$ we can pass to the canonical cylindric coordinates $(r,\varphi,x_3,p_r,p_\varphi=S,p_3)$ and to remove the cyclic coordinate $\varphi$.  We get the integrable reduced systems on $T^*(\R_+\times\R)$:
\begin{align*}
& H^s_1=\frac12 \big(p_r^2+p_3^2\big)+\frac{s^2}{2r^2}+g x_3,\\
& H^s_2=\frac12 (p_r^2+p_3^2)+\frac{s^2}{2r^2}+g x_3+\frac{1}{2}\rho(r^2+x_3^2)-\frac12\rho\ell^2.
\end{align*}
As Section \ref{sec3} indicates for $s=0$, the resulting system is non-integrable.
\end{exm}

It will be very interesting to find  natural examples of $C^0$  and $C^1$ gluing of integrable systems  with  first integrals that are of the class $C^0$ on $\Delta$.

\subsection*{Acknowledgements}
This research has been supported by the Project IntegraRS of the Science Fund
of Serbia, Mathematical Institute of the Serbian Academy of Sciences and Arts and the
Ministry for Education, Science, and Technological Development of Serbia, and the Simons
Foundation grant no. 854861.  We thank Dmitry Treschev for a valuable discussion.

\end{document}